\documentclass[11pt,twoside, leqno]{article}

\usepackage{amssymb,amsmath,amsfonts,amsthm,color,mathrsfs}
\usepackage[Symbol]{upgreek}
\usepackage{txfonts}
\usepackage[nottoc,notlot,notlof]{tocbibind}
\usepackage[active]{srcltx}
\usepackage{citeref}
\usepackage{hyperref}
\usepackage[dvipsnames, svgnames, x11names]{xcolor}

\usepackage{graphicx}
\usepackage{epsfig}

\usepackage{mathtools} 

\usepackage[dvipsnames]{xcolor} 

\definecolor{darkbrown}{RGB}{77, 45, 29}

\usepackage{txfonts}
\allowdisplaybreaks
\def\rr{{\mathbb R}}
\def\rn{{{\rr}^n}}

\def\fz{\infty}

\def\supp{{\mathop\mathrm{\,supp\,}}}
\def\dist{{\mathop\mathrm {\,dist\,}}}
\def\diam{{\mathop\mathrm {\,diam\,}}}
\def\abdiam{{\mathop\mathrm {\hspace{0.2mm} diam}}}

\def\boz{{\Omega}}

\def\r{\right}
\def\lf{\left}

\def\cdini{C^{1,\mathrm{Dini}}}

\newtheorem{thm}{Theorem}[section]
\newtheorem{lem}[thm]{Lemma}
\newtheorem{prop}[thm]{Proposition}
\newtheorem{cor}[thm]{Corollary}

\newtheorem{rem}[thm]{Remark}

\numberwithin{equation}{section}

\begin{document}
\arraycolsep=1pt
\author{Renjin Jiang, Tianjun Shen, Sibei Yang \& Houkun Zhang}
\title{{\bf Heat kernel estimates, fractional Riesz transforms and applications  on exterior domains}
 \footnotetext{\hspace{-0.35cm} 2010 {\it Mathematics
Subject Classification}. Primary  42B20; Secondary  42B37, 35K08, 35J08.
\endgraf{
{\it Key words and phrases: heat kernel, exterior domain, fractional Riesz transform}
\endgraf}}
\date{}}
\maketitle

\begin{center}
\begin{minipage}{11.5cm}\small
{\noindent{\bf Abstract}.
In this paper, we derive sharp two side heat kernel estimate on exterior $C^{1,1}$ domains in the plane, and  sharp
upper heat kernel bound on exterior $C^{1,\mathrm{Dini}}$ domains in $\mathbb{R}^n$, $n\ge 2$.  Estimates for Green's function
and Riesz potentials on exterior domains in the plane are also presented. Based on the heat kernel estimates, we show the boundedness
of the fractional Riesz transforms on exterior $C^{1,\mathrm{Dini}}$ domains in $\mathbb{R}^n$, $n\ge 2$. Some further applications
to product and chain rules and nonlinear Schr\"odinger equation are also presented.
}\end{minipage}
\end{center}
\vspace{0.2cm}
\tableofcontents

\section{Introduction} \hskip\parindent
%
%
%
%
%
%
%
%
%
In this paper, we consider the heat kernel and the fractional Riesz transform for the Dirichlet operator on certain exterior domain. Here a domain $\Omega:= \mathbb{R}^n \setminus \overline{U}$ is an exterior $C^{1,1}$ ($\cdini$) domain, if $U$ is a bounded $C^{1,1}$ ($\cdini$) domain. 
Moreover, we shall assume that $\Omega$ is connected, since the heat kernel on bounded domains are clear (cf. \cite{davies89,zha03}). 
Denote by $\Delta$ the Laplacian operator, and by $\Delta_\Omega$ the Dirichlet Laplacian on $\Omega$; see for instance \cite{KVZ16}.

We say a bounded domain $U \subset \mathbb{R}^n$ is a $C^{1,1}$ domain, if for each point $x_0 \in \partial U$ there exist $r >0$ and a $C^{1,1}$ function
(i.e., $C^1$ function whose derivatives are Lipschitz continuous)
$\gamma: \mathbb{R}^{n-1} \rightarrow \mathbb{R}$ such that
(upon relabeling and reorienting the coordinates axes if necessary)
in a new coordinate system $(x', x^n) = (x^1, \dots, x^{n-1}, x^n)$,
$$U \cap B(x_0,r) = \{ x\in B(x_0,r): x^n > \gamma(x')\};$$
see also \cite{DEK18,Ev10,FGS86}.
The classes of bounded $\cdini$ domains is defined analogously, requiring
that the defining functions $\gamma$ have first order derivatives which are uniformly
Dini continuous. Recall that a nonnegative measurable function $\omega:\, (0,a)\to \rr$ is a Dini function 
provided that 
$$\omega(t)\sim \omega(s)$$
whenever $t/2\le s\le t$ and $0<t<a$, and 
$$\int_0^t\frac{\omega(s)}{s}\,ds<\infty\,\ \forall\, t\in (0,a).$$
Here and in what follows, the \emph{symbol} $f\sim g$ means $f\lesssim g$ and $g\lesssim f$, which stands for $f\le Cg$ and $g\le Ch$, respectively, for some harmless positive constant $C$. A function $g$ defined on $\Omega$ is called uniformly Dini continuous, if for $r>0$, 
$$\rho_g(r):=\sup_{x\in\Omega} \sup_{y,z\in \Omega\cap B(x,r)}|g(y)-g(z)|$$
is a Dini function. Note that Dini continuity is stronger than continuity but weaker than H\"older continuity.

\subsection{Heat kernels and Riesz potentials}\hskip\parindent
Let us denote by $p(t,x,y)$ the heat kernel on $\mathbb{R}^n$,
that is by definition the minimal positive fundamental solution to
the heat equation on $\mathbb{R}^n$. We also denote by $p_{\Omega}(t,x,y)$
the Dirichlet heat kernel in $\Omega$ which is a minimal positive solution of the
following equation
$$
\begin{cases}
  \partial_t u(t,x) = \Delta u(x), & \, \forall \, x \in \Omega, t>0, \\
  u(x,t)= 0, & \, \forall \, x \in \partial \Omega, t>0,  \\
  u(x,0) = \delta_y & \, \forall \, x \in \Omega,
\end{cases}
$$
for any $y \in \Omega$, where $\delta_y$ is the Dirac function; see \cite{gsc02}.

Much effort has been devoted
to the precise estimation of heat kernels in various
Riemannian manifolds as well as Euclidean domains; see \cite{davies87,davies89,davies97,gsc02,gsa11,zha02,zha03}
and references therein.
For the   case of the exterior domain,
the upper and lower bounds of $p_\Omega$ for
 $x, y$ being away from the boundary of $\Omega$ are established by Grigor'yan and Saloff-Coste \cite{gsc02}.
The boundary behavior of the heat kernel was later discovered by Zhang \cite{zha03}. In particular, for an exterior $C^{1,1}$ domain
in $\rn$ with $n\ge 3$, it holds that

\vspace{0.2cm}
{\noindent {\bf{Theorem A. }}
{Let $\Omega \subset \mathbb{R}^n$ be an exterior $C^{1,1}$ domain, where $n \ge 3$. Then, for any $x, y \in \Omega$ and $t >0$,
\[
    p_{\Omega}(t,x,y)
    \sim \frac{1}{t^{n/2}}
\left(\frac{\rho(x)}{\sqrt t\wedge 1}\wedge 1\right)
\left(\frac{\rho(y)}{\sqrt t\wedge 1}\wedge 1\right)e^{-\frac{|x-y|^2}{ct}}.
\]}}
Here and in what follows, $\rho(x):= \dist(x,\Omega^c)$ for any $x \in \mathbb{R}^n$, $\Omega^c:=\rn\backslash\Omega$, and for any $a,b\in\mathbb{R}$,
$a\wedge b:=\min\{a,b\}$ and $a\vee b:=\max\{a,b\}$. 

To derive the upper and lower bound for the heat kernel,
their proof relies on the boundary Harnack inequality developed in \cite{FGS86}
and the two-sided estimate for the Green function, which requires the domain to be $C^{1,1}$ (see \cite[p. 335]{Bo00} and also \cite{gw82}).
By applying the method that takes advantage of harmonic weights for parabolic manifolds developed in \cite{gsc02} (see also \cite{gsa11}), we are able to give a sharp upper and lower bound of the heat kernel as follows.

\begin{thm} \label{main-heat-kernel}
Suppose that $\Omega \subset \mathbb{R}^2$ is an exterior $C^{1,1}$ domain. Then it holds for each $t>0$ and all $x,y\in\Omega$ that
$$p_\Omega(t,x,y) \sim   \frac{\log(e+\rho(x))\log(e+\rho(y))}{t\left(\log(e+\sqrt t)+\log(e+\rho(x)\vee\rho(y))\right)^2}
\left(\frac{\rho(x)}{\sqrt t\wedge 1}\wedge 1\right)
\left(\frac{\rho(y)}{\sqrt t\wedge 1}\wedge 1\right)e^{-\frac{|x-y|^2}{ct}}.
$$
\end{thm}
By \cite{gsc02}, to prove the above result, it is enough to consider only the case when at least one point $x$ or $y$ is near the boundary.
We shall first give a sharp upper bound of the heat kernel
on exterior $\cdini$ domains, whose proof works for general dimension $n\ge 2$.
Towards the lower bound on exterior $C^{1,1}$ domains in the plane, we first derive a lower bound for the first eigenfunction
on bounded $C^{1,1}$ domains by a lifting argument.
Then using the boundary Harnack inequality for parabolic equations from \cite{FGS86} gives the lower bound for the heat kernel, adopting a methodology similar to that presented in  \cite{zha03}.

We wish to point out that  \cite[Corollary 5.10 \ \& (5.6)]{gsa11} give a two-side estimate of the heat kernel in terms of positive harmonic weights with zero boundary, which is called
r\'eduite there. This harmonic weight  can be taken  as \(\phi-1\) where $\phi$ is as in \eqref{auxiliary-function} below. We believe it is also possible to derive the above theorem
from  \cite[Corollary 5.10 \ \& (5.6)]{gsa11} together with studying the boundary regularity of \(\phi-1\).
We adopt the current approach due to two reasons: one is that we want to give a unified approach to upper bound on $\cdini$ domains for all dimensions $n\ge 2$,
the other is that establishing the lower bound of the harmonic weight \(\phi-1\) near the boundary does not seem simpler than that of the first eigenfunction.

In particular, if $\Omega$ is an exterior $\cdini$ domain, we have the upper bound.
\begin{thm} \label{cor:upper-bound}
  Suppose that $\Omega \subset \mathbb{R}^n$ is an exterior $C^{1,\mathrm{Dini}}$ domain, $n \ge 2$.  Then it holds  for each $t>0$ and all $x,y\in\Omega$ that
  $$p_\Omega(t,x,y) \lesssim
  \begin{dcases}
  \frac{\log(e+\rho(x))\log(e+\rho(y))}{t\left(\log(e+\sqrt t)+\log(e+\rho(x)\vee\rho(y))\right)^2}\left(\frac{\rho(x)}{\sqrt t\wedge 1}\wedge 1\right)\left(\frac{\rho(y)}{\sqrt t\wedge 1}\wedge 1\right)e^{-\frac{|x-y|^2}{ct}}, & \ n=2,\\
\frac{1}{t^{n/2}}
\left(\frac{\rho(x)}{\sqrt t\wedge 1}\wedge 1\right)
\left(\frac{\rho(y)}{\sqrt t\wedge 1}\wedge 1\right)e^{-\frac{|x-y|^2}{ct}}, &\ n\ge 3.
\end{dcases}
$$
\end{thm}

Compared to the parabolic problem, the elliptic counter part also attracts significant attention;
see \cite{Bo00,gw82,zhao86} and references therein.
This estimate plays a crucial role in potential analysis and mathematical physics.
For instance, the well-known three-G theorem can be derived from it (see e.g. \cite[(19)]{Bo00}).
Making use of the aforementioned heat kernel estimates, we establish a two side estimate for the Green function on
the exterior domains in the following way.
\begin{thm}\label{green-function}
Let $\Omega \subset \mathbb{R}^2$ be an exterior $C^{1,1}$ domain.
Denote by $\Gamma_\Omega(\cdot,\cdot)$ the Green function on $\Omega$. Then it holds for all $x,y\in\Omega$ that
$$
\Gamma_\Omega(x,y)\sim
\begin{dcases}
 1 + \log \frac{\rho(x)\wedge \rho(y)}{|x-y|},  &|x-y|< \rho(x)\wedge \rho(y)\wedge 1, \\
\frac{\left(\rho(x)\wedge 1\right)\left(\rho(y)\wedge 1\right)}{|x-y|^2\wedge 1}\log (e+\rho(x)\wedge \rho(y)), &\ \mbox{otherwise}.
\end{dcases}
$$
Moreover, if $\Omega$ is an exterior $\cdini$ domain, then the upper bound
of $\Gamma_\Omega$ still holds.
\end{thm}

Integrating the estimates in Theorem \ref{cor:upper-bound}, we obtain upper bound
 for the Riesz potential.

\begin{prop}\label{riesz-potential-prop}
Let $\Omega $ be an exterior $\cdini$ domain in $\rn$, $n\ge 2$. Then for $0<s<n$, it holds  for all $x,y\in\Omega$  that
$$(-\Delta_{\Omega})^{-\frac{s}{2}}(x, y)\lesssim
\begin{dcases}\frac{1}{|x-y|^{2-s}}
\frac{\log(e+\rho(x)\wedge \rho(y))}{\log(e+\rho(y)\vee\rho(x))}\left(\frac{\rho(x)}
{|x-y| \wedge 1} \wedge 1\right)
\left(\frac{\rho(y)}{|x-y| \wedge 1} \wedge 1\right), &\ n=2,\\
\frac{1}{|x-y|^{n-s}}\left(\frac{\rho(x)}
{|x-y| \wedge 1} \wedge 1\right)
\left(\frac{\rho(y)}{|x-y| \wedge 1} \wedge 1\right), &\ n\ge 3.
\end{dcases}
$$
\end{prop}
By using the subordinate formula one can also give two side estimate
for the kernel of semigroup $e^{-t(-\Delta_\Omega)^s}$, which we leave to interested readers.
We note that, for the fractional Dirichlet  non-local Laplacian operator with zero
exterior condition $(-\Delta)_\Omega^s$, two-side estimates for the heat kernel of $e^{-t(-\Delta)_\Omega^s}$ were established in \cite{cks10}.

\subsection{Fractional Riesz transform and applications}\hskip\parindent
With the above estimates of heat kernels and Riesz potentials at hand, we are able to move to the study of fractional Riesz transforms.
For $s\ge 0$ and $1<p<\infty$, let $\dot{H}^{s,p}_D(\Omega)$  be the completion
of $C^\infty_c(\Omega)$ under the norms
$$\|f\|_{\dot{H}^{s,p}_D(\Omega)}:=\left\|(-\Delta_\Omega)^{s/2}f\right\|_{L^p(\Omega)}.$$
Similarly, $\dot{H}^{s,p}(\rn)$ is defined with $\Omega$ replaced by $\rn$ in the above definition.
The Sobolev space $\dot{H}_0^{s,p}(\Omega)$ is defined as the completion of  $C^\infty_c(\Omega)$ in
$\dot{H}^{s,p}(\rn)$. The inhomogeneous Sobolev spaces  ${H}^{s,p}_D(\Omega)$, ${H}_0^{s,p}(\Omega)$ are then defined as the
completion of $C^\infty_c(\Omega)$ under the norms $\left\|(-\Delta_\Omega)^{s/2}f\right\|_{L^p(\Omega)}+\|f\|_{L^p(\Omega)}$
and  $\left\|(-\Delta)^{s/2}f\right\|_{L^p(\rn)}+\|f\|_{L^p(\Omega)}$, respectively.

We are interested to see whether the two spaces $\dot{H}^{s,p}_0(\Omega)$ and  $\dot{H}^{s,p}_D(\Omega)$ are equivalent or not.
The study of such equivalence has a deep root in  both harmonic analysis and PDEs, see \cite{Fu67,gr67,hs09,JL24,JY24,KVZ16,kvz16b,lsz12,se72} for instance.

For the case $s=1$, by the well-known property of classical Riesz transforms
$$\left\|\nabla (-\Delta)^{-1/2}\right\|_{L^p(\rn)\to L^p(\rn)}\le C(n,p)<\infty$$
and the reverse inequality
$$
\left\|(-\Delta)^{1/2}f\right\|_{L^p(\rn)} \lesssim C(n,p)
\left\|\,|\nabla f|\,\right\|_{L^p(\rn)},
\quad \forall\, f \in C_c^{\infty}(\rn),
$$
for $1<p<\infty$ (see e.g. \cite{stw71}), $\dot{H}_0^{1,p}(\Omega)$
coincides with the usual Sobolev space $\dot{W}_0^{1,p}(\Omega)$, which is defined via the distributional gradient.
So the question of equivalence between $\dot{H}^{1,p}_0(\Omega)$ and  $\dot{H}^{1,p}_D(\Omega)$  reduces to
the boundedness of the Riesz transform, i.e., does the inequality
$$C^{-1}\|\nabla f\|_{L^p(\Omega)}\le \left\|(-\Delta_\Omega)^{1/2}f\right\|_{L^p(\Omega)}\le C\|\nabla f\|_{L^p(\Omega)}$$
holds? We refer the reader to \cite{acdh04,cch06} for pioneering  studies on manifolds.
In the case of exterior domains, the question has been addressed by \cite{hs09,lsz12}
in case of radially symmetric functions, by \cite{KVZ16} for domains outside convex obstacles,
and characterization of boundedness has recently been established in \cite{JL24}.

For the fractional case, $0<s<2$ and $s\neq 1$, Killip et al. \cite{KVZ16} also proved the boundedness of the fractional Riesz transform
in the sharp range for exterior domains outside convex obstacles in $\rn$, $n\ge 3$. With the help of our heat kernel and Riesz potential estimates (cf. Theorem \ref{cor:upper-bound} and Proposition \ref{riesz-potential-prop}),
following an approach analogous to that in \cite{KVZ16}, we establish the boundedness of the fractional Riesz transform as follows.
This extends the result in \cite{KVZ16} by generalizing it from exterior convex domains to exterior $\cdini$ domains,
and filling the left part for the planar case.

\begin{thm}\label{main-1}
Suppose that $\Omega\subset\rn$ is a domain outside a bounded $C^{1,\mathrm{Dini}}$ obstacle, $n\ge 2$. Let $1<p<\infty$.
Then for $0<s<\min\{n/p,1+1/p\}$, it holds for all $f\in C^\infty_c(\Omega)$ that
$$\left\|(-\Delta )^{s/2}f\right\|_{L^{p}(\rn)}\sim \left\|(-\Delta_\Omega)^{s/2}f\right\|_{L^{p}(\Omega)}.$$
Consequently, for such $s$ and $p$, the two spaces $\dot{H}^{s,p}_0(\Omega)$ and  $\dot{H}^{s,p}_D(\Omega)$
coincides with equivalent norms.
\end{thm}
\begin{rem}\rm
(i) It was shown in \cite{hs09} that the Riesz
transform $\nabla(-\Delta_\Omega)^{-1/2}$ on the exterior of the unit ball is \emph{not} bounded on $L^p$ for $p\in(2,\fz)$ if $n=2$,
and $p\in[n,\fz)$ if $n\ge3$; see also \cite[Remark 3.1]{JL24}. For general fractional order, \cite[Proposition 7.1 \& Proposition 7.2]{KVZ16} showed that  for $n/p\le s<2$ or $s\ge 1+\frac {1}{p}$, $n\ge 3$,
the equivalence
$$\left\|(-\Delta )^{s/2}f\right\|_{L^{p}(\rn)}\sim \left\|(-\Delta_\Omega)^{s/2}f\right\|_{L^{p}(\Omega)}$$
does not hold.

Similarly to \cite[Proposition 7.1 \& Proposition 7.2]{KVZ16},
and by using the heat kernel estimate (Theorem \ref{main-heat-kernel}),
we can
show that for $n=2$, $2/p< s<2$  the conclusion of Theorem \ref{main-1} does not hold; see Remark \ref{unboundedness} below.
Note that for $n=2$, $1<p<\infty$, it holds that $2/p<1+1/p$.

(ii) The borderline case $s=n/p$, $n=2$, is missing from Theorem \ref{main-1} or from part (i) of this remark.
Note that for $s=1$ and $p=2$, Theorem \ref{main-1} is trivially true. We therefore expect that Theorem \ref{main-1}
should also be true for $0<s<2$ and $s=2/p$.
\end{rem}

With the aid of Theorem \ref{main-1},
we obtain fractional  product and chain rules on exterior $\cdini$ domains.
These following two corollaries are straightforward combinations of Theorem \ref{main-1} and the results in the Euclidean case,
with relevant details for the latter provided in  \cite[Propositions 3.1 and 3.3]{CW91}; see also \cite{Ta20}.

\begin{cor}\label{c1.1}
	Let $\Omega \subset \mathbb{R}^n$ be an exterior $\cdini$ domain, $n \ge 2$. Assume that $G \in C^1(\mathbb{C})$, $1 < p, p_1, p_2 < \infty$ and
	$0< s < \min \{1 + 1/p_2, n/p_2\}$, where $1/p = 1/p_1 + 1/p_2$. Then it holds that
\[
\left\|(-\Delta_\Omega)^{s/2} G(f)\right\|_{L^p(\Omega)}
	\lesssim \left\|G'(f)\right\|_{L^{p_1}(\Omega)}
	\left\|(-\Delta_\Omega)^{s/2} f \right\|_{L^{p_2}(\Omega)},
\]
uniformly for any $f \in C_c^\infty(\Omega)$.
\end{cor}

\begin{cor}
Let $\Omega \subset \mathbb{R}^n$ be an exterior $\cdini$ domain, $n \ge 2$. Then for all $f, g \in C_c^{\infty}(\Omega)$, it holds that
\[
\left\|(-\Delta_\Omega)^{s/2} (fg)\right\|_{L^p(\Omega)}
\lesssim\left\|(-\Delta_\Omega)^{s/2} f\right\|_{L^{p_1}(\Omega)}
\|g \|_{L^{p_2}(\Omega)}
+\|f \|_{L^{q_1}(\Omega)}\left\|(-\Delta_\Omega)^{s/2}g\right\|_{L^{q_2}(\Omega)},
\]
where the exponents satisfy $1 < p, p_1,q_2 <\infty$,
$1 < p_2, q_1 \leq \infty$, $1/p = 1/p_1 + 1/p_2 = 1/q_1 + 1/q_2$,
and $0 < s < \min\{1 + 1/p_1, 1+ 1/q_2, n/p_1, n/q_2\}$.
\end{cor}

With the chain rule, we can also extend the local well-posedness of the nonlinear Schr\"odinger equation (NLS), i.e.,
\begin{equation}\label{nls-2d-intro}
i\partial_t u=-\Delta_{\Omega} u \pm |u|^pu
\quad
\text{with }
u(x,0)=u_0(x)
\quad \text{and}\quad
u(x,t)|_{\partial\Omega}=0,
\end{equation}
to the planar case.
\begin{thm}\label{local-NLS}
Let $s\in (0,1)$, $p:=4/(2-2s)$, $r:=({6-2s})/({1+s})$, and $q:=({3-s})/({1-s})$. Let $\Omega\subset\rr^2$ be the exterior of a smooth compact strictly convex obstacle. There exists $\eta>0$ such that if $u_0\in H^s_D(\Omega)$ satisfies
$$\left\|(-\Delta_\Omega)^{\frac s2}e^{it\Delta_\Omega}u_0\right\|_{L_t^{q}L_x^r(I\times\Omega)}\le \eta$$
for some time interval $I$ containing $0$, then there is a unique strong $C_t^0\dot{H}_D^s(I\times\Omega)$  solution to the equation \eqref{nls-2d-intro}, and it holds that
$$\left\|(-\Delta_\Omega)^{\frac s2}u\right\|_{L_t^{q}L_x^r(I\times\Omega)}\lesssim \eta.$$
\end{thm}

The paper is organized as follows. In Section \ref{s2}, we provide the estimate for the heat kernel, and prove Theorems \ref{main-heat-kernel} and \ref{cor:upper-bound}. In Section \ref{s3},
we provide the estimates for Green's function and Riesz potentials,
and prove Theorem \ref{green-function} and Proposition \ref{riesz-potential-prop}.
In Section \ref{s4}, we provide the proof of Theorem \ref{main-1}, and in Section \ref{s5}, we prove Theorem \ref{local-NLS}.

Throughout the paper,
the letters $C$, $c$, $c'$, $c''$  denote positive constants
which are independent of the
main parameters, but may vary from line to line.
When the value of a constant is significant, it will be explicitly stated.
The symbol $A \lesssim B$ means that $A \leq C B$, and $A \sim B$ means $c A \leq B \leq C A$, for some harmless constants $c, C>0$. For any measurable subset $E$ of $\rn$, we denote by $E^c$ the set $\rn\setminus E$. Furthermore, for any $q\in[1,\fz]$, we denote by $q'$ its conjugate exponent, that is, $1/q+1/q'= 1$.

\section{Heat kernel estimate}\label{s2} \hskip\parindent
In this section,
we aim to derive the two-sided bound for the Dirichlet heat kernel on exterior domains in the planer case $\mathbb{R}^2$.
Recall that for the higher-dimensional case $\mathbb{R}^n$, where $n \ge 3$, Zhang \cite{zha03} discovered that for an exterior $C^{1,1}$ domain
$\Omega \subset \mathbb{R}^n$,
the heat kernel satisfies
\begin{align}\label{heat-n3-1}
p_\Omega(t,x,y)\sim t^{-n/2}\left(\frac{\rho(x)}{\sqrt t\wedge 1}\wedge 1\right)\left(\frac{\rho(y)}
{\sqrt t\wedge 1}\wedge 1\right)e^{-\frac{|x-y|^2}{ct}}.
\end{align}
See also \cite{zha02} for bounded $C^{1,1}$ domains.

However, boundary behaviors of the heat kernel on exterior domains are not clear for the planar case, up to present.
In fact, the estimate in \eqref{heat-n3-1} does not hold for the case $n=2$,  as the result from \cite{gsc02} shows that, on
a domain outside the unit ball, $\rr^2\setminus \overline{B(0,1)}$,  it holds for any $x,y$ with $|x|,|y|>C>1$ and $t>0$ that
\begin{equation*}
p_\Omega(t,x,y)\sim \frac{\log|x|\log|y| }{t(\log(1+\sqrt t)+\log|x|)(\log(1+\sqrt t)+\log|y|)}e^{-\frac{|x-y|^2}{4t}}.
\end{equation*}

Note that deriving the sharp heat kernel estimates typically relies on
the local comparison principle from \cite{FGS86}, the two-sided estimates for Green's function
from \cite{gw82,zhao86}
or the fact that the first eigenfunction $\phi$ is comparable to the distance function
on a bounded $C^{1,1}$ domain shown in \cite{davies87}.
While the local comparison theorem (see \cite{FGS86})
remains applicable in the planar case,
the validity of the other two properties remains unclear.

In what follows, we will primarily focus on the planar case.
However, some of our results are also valid in higher dimensions ($n \geq 3$),
and these cases will be explicitly pointed out.

To obtain the sharp bound for the heat kernel in two dimensional space,
it is reasonable to take advantage of results concerning non-parabolic manifolds.
Following the argument in \cite{gsc02}, we convert the parabolic manifold $(\mathbb{R}^2, dx)$
into a non-parabolic manifold $(\mathbb{R}^2, \phi^2 dx)$,
where $\phi$ is a harmonic function to be defined later.

For any positive smooth function $h$ on $\mathbb{R}^2$, the Laplace operator $\Delta^h$ of
$(\mathbb{R}^2, h^2 dx)$ is given by
$$
 \Delta^h f = h^{-2} \text{div}(h^2 \nabla f).
$$
Let
$p^h$ denote the heat kernel associated with $\Delta^h$,
and let $p_U^h$ be the Dirichlet heat kernel in $U$ associated with $\Delta^h$, where $U \subset \mathbb{R}^2$ is an
open set.

There is a tight connection between the heat kernels $p_U$ and $p_U^h$, known as Doob's transform.

\begin{prop}
Let $h$ be a positive function on $\mathbb{R}^2$.
Suppose that $\Delta u=0$ in an open set $U\subset \mathbb{R}^2$.
Then the heat kernel $p_U$ and $p^{h}_U$
are related by
\begin{align*}
p_U(t,x,y)=h(x)h(y)p_U^{h}(t,x,y), \quad \forall \, x,y\in U, \, t>0.
\end{align*}
\end{prop}
\begin{proof}
  For the proof, we refer to \cite[Proposition 4.2]{gsc02}.
\end{proof}

In fact, there exists a specific harmonic function $\phi$ such that $(\mathbb{R}^2, d\mu)$ becomes a non-parabolic manifold, where $d\mu = \phi^2 dx$.

\begin{prop} \label{prop:harmonic-es}
The following statements are valid.
\begin{itemize}
\item [(i)] Let $U:=\mathbb{R}^2 \setminus \overline{V}$, where $V \subset \mathbb{R}^2$ is a bounded open set.
    Then there exists a positive smooth function $\phi$ on $\mathbb{R}^2$
    that is harmonic in $U$ and admits the estimate
\[
\phi(x) \sim  \log(e + \rho(x)), \quad \forall \, x \in \mathbb{R}^2.
\]
    \item [(ii)] The weighted manifold $(\mathbb{R}^2, d\mu)$ with
    $d\mu = \phi^2 dx$ is non-parabolic and the heat kernel $p^{\phi}$ satisfies
\begin{equation*}
p^{\phi}(t,x,y)\sim \frac{1}{\sqrt{\mu(x,\sqrt t)\mu(y,\sqrt t)}}e^{-\frac{|x-y|^2}{ct}},
\quad \forall \, x,y \in \mathbb{R}^2, \, t >0.
\end{equation*}
Here, and thereafter, $\mu(x,\sqrt t):=\mu(B(x,\sqrt t))=\int_{B(x,\sqrt t)}\phi^2\,dy$.
Moreover, the measure $\mu$ is a doubling measure on $\mathbb{R}^2$, and satisfies
\[
\mu(B(x,r)) \sim r^2 [\log(e + r) + \log(e + \rho(x))]^2,\quad \forall \, x \in \mathbb{R}^2, \, r>0.
\]
\end{itemize}
\end{prop}
\begin{proof}
For the proof, we refer to \cite[Lemmas 4.7 and 4.8]{gsc02}.
\end{proof}

\begin{rem}\rm

In fact, the aforementioned method for converting a parabolic manifold into a non-parabolic one is applicable not only for $\mathbb{R}^2$,
but also for manifolds  satisfying the parabolic Harnack inequality and the relatively connected annuli condition. For further details,
see \cite[Section 4]{gsc02}. Moreover, the above statements rely solely on the assumption that $\phi$ is a harmonic function satisfying
certain growth conditions. In the special case of our paper, where
the open set $\Omega \subset \mathbb{R}^2$ is an exterior $\cdini$ domain, we can explicitly construct $\phi$ as follows.

Without loss of generality, we may assume that the spatial origin $0 \in \Omega^c$.
Define
$$u_0(x):=\frac{1}{2\pi\sigma(\partial\Omega)}\int_{\partial\Omega}\ln |y-x|\,d\sigma(y),$$
where $d\sigma$ denotes the surface measure on $\partial\Omega$.
Then it holds that $\Delta u_0=\frac{1}{\sigma(\partial\boz)}\delta_{\partial\boz}$
and we have $u_0\in \dot{W}^{1,p}(\rr^2)$ for any $p\in (2,\infty)$ but $u_0\notin \dot{W}^{1,2}(\rr^2)$.
Otherwise, since $\Delta u_0\in \dot{W}^{-1,2}(\rr^2)$ and $1\in \dot{W}^{1,2}(\rr^2)$,
it would follow the contradiction $1=\langle\Delta u_0,1\rangle=0$.

Next, let $u_1$ denote the
unique solution in $\dot{W}^{1,2}(\boz)\cap\dot{W}^{1,p}(\boz)$
to the boundary value problem
\begin{equation*}\label{e2.9a}
\lf\{\begin{array}{ll}
-\Delta u_1=0,\ \ &\text{in}\ \boz,\\
u_1=u_0,\  &\text{on}\  \partial\boz.
\end{array}\r.
\end{equation*}
See \cite[Theorem 2.7 \& Remark 2.8]{agg97} or \cite[Proposition 3.3]{JY24} for instance.

By the maximal principle, either $u_0-u_1<0$ or $u_1-u_0<0$ in $\Omega$.
Assume without loss of generality that $u_0-u_1<0$.
For the $\cdini$ domain $\Omega$, the function $u_0-u_1$
exhibits the asymptotic behavior
\begin{equation} \label{eq:u0u1}
(u_0-u_1)(x)=-c_0\ln|x|+O(|x|^{-1})
\end{equation}
as $|x|\to\fz$, for some $c_0>0$; see \cite{Ver84} and \cite[Remark 5.5]{sw23}.

Finally, define the auxiliary function
\begin{align}\label{auxiliary-function}
\phi(x)=
\begin{cases}
1+\frac{1}{c_0}(u_1-u_0)(x), &\,\forall\,x\in{\Omega},\\
1,&\, \forall \, x\in \rr^2\setminus\Omega.
\end{cases}
\end{align}

By \eqref{eq:u0u1}, we find that there exists a positive constant $M > 2 \diam(\Omega^c)$ such that,
for any $|x| \ge M$, it holds that $(u_1 - u_0)(x) \sim \log(e+ |x|)$.
Thus, for the case $|x| > M$, it holds that $\rho(x) \ge |x|-\diam(\Omega^c) \gtrsim |x|$ and
$\rho(x) \sim |x|$, which yields that
$\phi(x) \sim \log(e + \rho(x))$.
For the case $|x| < M$, it follows from
the facts $\phi(x) \ge 1$ and $\rho(x) \leq M$ that
$\phi(x) \gtrsim \log(e + \rho(x))$.
Besides, by the fact that $\phi$ is a continuous function,
we conclude that $\phi(x) \lesssim \log(e + \rho(x))$ for the case $|x| < M$. Therefore,
for any $x\in\rn$, $\phi(x) \sim \log(e+\rho(x))$.
\end{rem}

\subsection{Upper bound of the heat kernel} \hskip\parindent
Let us extend the heat kernels $p_\Omega(t,x,y)$  and $p_\Omega^\phi(t,x,y)$ to $\rr_+\times\rr^2\times\rr^2$ by
defining them to be zero for any $x \in \Omega^c$ or $y \in \Omega^c$.
We have the following upper bound for the heat kernel $p_\Omega^\phi(t,x,y)$.
\begin{thm}\label{heat-upper}
Let $\Omega$ be an exterior $\cdini$ domain in $\rr^2$.
There there exists a constant $C >0$ such that
\begin{equation*}
p^\phi_\Omega(t,x,y)\le \frac{C}{\sqrt{\mu(x,\sqrt t)\mu(y,\sqrt t)}}\left(\frac{\rho(x)}{\sqrt t\wedge 1}\wedge 1\right)
\left( \frac{\rho(y)}{\sqrt t\wedge 1}\wedge 1\right) e^{-\frac{|x-y|^2}{ct}}.
\end{equation*}
for all $x, y \in \mathbb{R}^2$, and any $t >0$.
\end{thm}

\begin{proof}
\textbf{Step 1.} Since $\phi$ is a Lipschitz function and $\Omega$ is an exterior $\cdini$ domain, by  \cite[Theorem 1.3]{DEK18},
any solution $u$ to the Poisson equation
$$\Delta^\phi_{\Omega} u=g,$$
with $g\in L^p(\Omega)$, where $p>n$, belongs to $C^1_{\mathrm{loc}}(\Omega)$ and satisfies that, for any $x_0 \in \Omega$,
\begin{equation}\label{poisson}
\|\nabla u\|_{L^\infty(B(x_0,r)\cap\Omega)}
\le \frac{C}{r}
\fint_{B(x_0,2r)\cap \Omega}|u|\,d\mu(x)
+Cr\left(\fint_{B(x_0,2r)\cap \Omega}|g|^{p}\,d\mu(x)\right)^{1/p};
\end{equation}
see the proof of \cite[Proposition 2.7]{DEK18}.

Since the heat kernel $p^\phi(t,x,y)$ satisfies the Gaussian upper bound
(see Proposition \ref{prop:harmonic-es}), the maximal principle implies that for any $x,y \in \mathbb{R}^2$ and $t >0$,
\begin{equation} \label{eq:rough-es}
p_\Omega^\phi(t,x,y)\le p^\phi(t,x,y)\le  \frac{C}{\mu(y,\sqrt t)} e^{-\frac{|x-y|^2}{ct}}.
\end{equation}
Moreover, it follows from \cite[Theorem 4]{davies97} that the time gradient of $p_\Omega^\phi$ satisfies that, for any $x, y \in \mathbb{R}^2$ and $t >0$,
\begin{equation*}
|\partial_t p^\phi_\Omega(t,x,y)|\le  \frac{C}{t\mu(y,\sqrt t)} e^{-\frac{|x-y|^2}{ct}}.
\end{equation*}

\textbf{Step 2.}  For  all $x, y \in \Omega$ and $t >0$, we apply \eqref{poisson} to the equation $\Delta^\phi_{\Omega} p_\Omega^\phi(t,\cdot,y)=\partial_t p_\Omega^\phi(t,\cdot,y)$ in the ball $B(x, r)$
with $r=\sqrt t \wedge 1$,
which gives that
\begin{align*}
|\nabla_x p_\Omega^\phi(t,x,y)|
    \leq
    \frac{C}{r} \fint_{B(x, 2r) \cap \Omega} p_\Omega^\phi(t, z,y) \, d\mu(z)
      + C r\left(\fint_{B(x,2r)\cap \Omega} |\partial_t p_\Omega(t, z,y)|^p \, d\mu(z) \right)^{1/p}.
\end{align*}
For the case $|x-y| \ge 4r$, one has
\[
    |y-z| \ge |x-y| - |x-z| \ge |x-y| - |x-y| /2 = |x-y|/2,
\]
which further implies that
\[
    |\nabla_x p_\Omega^\phi(t,x,y)| \leq \frac{C}{\sqrt t \wedge 1} \frac{1}{\mu(y,\sqrt t)} e^{-\frac{|x-y|^2}{ct}}.
\]
For the case $|x-y| \leq 4 r$, we have
\[
    e^{-\frac{|x-y|^2}{ct}} \sim 1,
\]
which also implies that
\begin{equation} \label{eq:gradient-es}
    |\nabla_x p_\Omega^\phi(t,x,y)| \leq \frac{C}{\sqrt t \wedge 1} \frac{1}{\mu(y,\sqrt t)} e^{-\frac{|x-y|^2}{ct}}.
\end{equation}
By symmetry, the above inequality also holds with $x$ replaced by $y$.

\textbf{Step 3.}
In case of $\rho(y)<\sqrt t \wedge 1$, since $\partial \Omega$ is compact and $C^{1}$-regular, we can find $y_0\in\partial\Omega$ such that
$$|y-y_0|=\rho(y).$$
Let $\ell_{yy_0}$ be the geodesic connecting $y$ to $y_0$.
Using the fact that $p^\phi_\Omega(t,x,y_0)=0$, \eqref{eq:gradient-es} and the gradient theorem, we see that
\begin{align*}
  p_\Omega^\phi(t,x,y)=p_\Omega^\phi(t,x,y)-p_\Omega^\phi(t,x,y_0)
  \le \int_{\ell_{yy_0}} \frac{C}{\sqrt t \wedge 1}\frac{1}{\mu(x,\sqrt t)}
e^{-\frac{|x-z|^2}{ct}}\,d\ell(z).
\end{align*}
If $|x-y| \leq 2 \rho(y) \leq 2 \sqrt t$, we have
\[
    e^{-\frac{|x-y|^2}{ct}}  \sim 1,
\]
which implies that
\[
    p_\Omega^\phi(t,x,y)
    \le C \frac{\rho(y)}{\sqrt t \wedge 1}\frac{1}{\mu(y,\sqrt t)} e^{-\frac{|x-y|^2}{ct}}.
\]
If $|x-y| \ge 2 \rho(y) \ge 2|z-y|$, we have
\[
    |x-z| \ge |x-y| - |y-z| \ge  |x-y|/2,
\]
which also implies that
\[
    p_\Omega^\phi(t,x,y)
    \le C \frac{\rho(y)}{\sqrt t \wedge 1}\frac{1}{\mu(y,\sqrt t)} e^{-\frac{|x-y|^2}{ct}}.
\]

By symmetry of the heat kernel, if $\rho(x)<\sqrt t \wedge 1$, we also have
\begin{equation*}
  p_\Omega^\phi(t,x,y)\le C \frac{\rho(x)}{\sqrt t \wedge 1}\frac{1}{\mu(y,\sqrt t)} e^{-\frac{|x-y|^2}{ct}}.
\end{equation*}
Finally, if $\rho(y)<\sqrt t \wedge 1$ and $\rho(x)<\sqrt t \wedge 1$, it holds that
\begin{align*}
p^\phi_\Omega(t,x,y)
& = \int_{\Omega} p_\Omega(t/2,x,z) p_\Omega(t/2,z,y)\,dz\nonumber\\
& \le   C \frac{\rho(x)}{\sqrt t \wedge 1}\frac{1}{\mu(y,\sqrt t)} \frac{\rho(y)}{\sqrt t \wedge 1}\frac{1}{t} \int_{\Omega}
e^{-\frac{|z-y|^2+|y-z|^2}{ct}}
\,dz \nonumber \\
&\le  C \left(\frac{\rho(x)}{\sqrt t \wedge 1}\frac{\rho(y)}{\sqrt t \wedge 1} \right) \frac{1}{\mu(y,\sqrt t)} e^{-\frac{|x-y|^2}{ct}},
\end{align*}
which together with \eqref{eq:rough-es} gives that for any $x,y\in \mathbb{R}^2$ and $t > 0$,
$$p^\phi_\Omega(t,x,y)\le  \frac{C}{\sqrt{\mu(x,\sqrt t)\mu(y,\sqrt t)}}\left(\frac{\rho(x)}{\sqrt t\wedge 1}\wedge 1\right)\left( \frac{\rho(y)}{\sqrt t\wedge 1}\wedge 1\right) e^{-\frac{|x-y|^2}{ct}}.$$
This completes the proof.
\end{proof}

\begin{rem}\label{rem-heat-n3}\rm
The same proof shows that on an exterior $\cdini$ domain $\Omega\subset\rn$, where $n\ge 3$,
there exists a constant $C >0 $ such that
\begin{equation*}
p_\Omega(t,x,y)\le \frac{C}{t^{n/2}}\left(\frac{\rho(x)}{\sqrt t\wedge 1}\wedge 1\right)\left( \frac{\rho(y)}{\sqrt t\wedge 1}\wedge 1\right) %
e^{-\frac{|x-y|^2}{ct}},
\end{equation*}
for all $x, y \in \mathbb{R}^n$ and any $t >0$.
\end{rem}

\subsection{Lower bound of the heat kernel} \hskip\parindent
To derive the sharp heat kernel estimates, one typically relies on the local comparison principle for parabolic equations from \cite{FGS86}, the two-sided bounds for Green's function from \cite{gw82,zhao86}
or the fact that the first eigenfunction $\phi$ is comparable to the distance function on a bounded $C^{1,1}$ domain from \cite{davies87}. While the local comparison theorem (\cite{FGS86})
works for the planar case, the other two properties remains unclear in this setting.
Deriving two-sided bounds for the Green's function appears to be rather difficult.
However, we provide a proof for the two-sided bound of the first eigenfunction as follows.

\subsubsection{Lower bound of the first eigenfunction}\hskip\parindent
We need sharp estimate of the first eigenfunction for planar domains. Recall that the result for domains
on higher-dimensional Euclidean space ($n\ge 3$) was proved by Davies \cite{davies87}.
\begin{thm}\label{lower-eigenfunction}
Let $\Omega\subset\rr^2$ be a bounded $C^{1,1}$ domain. Let $\phi_\Omega$ be the first eigenfunction of $\Delta^\phi_{\Omega}$.
Then there exists a constant $C=C(\Omega)>0$ depending on $\Omega$ such that, for any $x\in\Omega$,
$$\frac 1C\rho(x)\le \phi_\Omega(x)\le C\rho(x).$$
 \end{thm}
 \begin{proof}
We normalise $\phi_\Omega$ such that $\|\phi_\Omega\|_{L^2(\Omega)}=1$. We first estimate the upper bound of $\phi_\Omega$.
Notice that $\|\phi_\Omega\|_{W^{1,2}(\Omega)} \leq C \|\phi_{\Omega}\|_{L^2(\Omega)}=C$.
By this and the Sobolev embedding theorem (see, for instance, \cite{Mo70}), we conclude that, for any $q\in(1,\infty)$,
 \begin{equation} \label{eq:emb}
 \|\phi_{\Omega}\|_{L^q(\Omega)}
 \leq C(\Omega) \|\phi_{\Omega}\|_{W^{1,2}(\Omega)}
\leq C(\Omega).
 \end{equation}
Moreover, since $\phi_\Omega \in C^{1,1}(\overline{\Omega})$, combining the estimates \eqref{poisson} and \eqref{eq:emb} further yields that,
for any $x\in\Omega$,
$$\phi_\Omega(x)\le C\|\nabla \phi_\Omega \|_{L^\infty(\Omega)}\rho(x)\le C(\Omega)\rho(x).$$

For the lower bound, we employ the standard technique of lifting $\phi_\Omega$ to a harmonic function $v$ on $\Omega\times\rr$,
defined as $v(x,t):=\phi_\Omega(x) e^{\sqrt\lambda t}$, where $\lambda$ is the first eigenvalue of $\Omega$, i.e.,
 $$-\Delta^\phi_{\Omega}\phi_\Omega=\lambda\phi_\Omega.$$

Let $R_0:=\mathrm{diam}(\Omega)$. Since $\Omega$ is a $C^{1,1}$ domain, we can choose a $C^{1,1}$ domain $\tilde \Omega$ in $\rr^3$ such that
$$\Omega\times (-2R_0,2R_0)\subset\tilde\Omega \subset \Omega\times (-4R_0,4R_0).$$
Let $\Gamma(\tilde x,\tilde y)$ be the Green function of the elliptic operator
$$\Delta_{\tilde\Omega}:= \frac{\partial^2}{\partial t^2} + \Delta^{\phi}$$
with Dirichlet boundary condition on $\tilde \Omega$. 
It follows from \cite{gw82,zhao86} (see also \cite[Theorem A]{zha03}) that
$$\frac 1C \left(\frac{\rho_{\tilde\Omega}(\tilde x)\rho_{\tilde\Omega}(\tilde y)}{|\tilde x-\tilde y|^2}\wedge 1\right)\frac{1}{|\tilde x-\tilde y|}\le\Gamma(\tilde x,\tilde y)\le C\left(\frac{\rho_{\tilde\Omega}(\tilde x)\rho_{\tilde\Omega}(\tilde y)}{|\tilde x-\tilde y|^2}\wedge 1\right)\frac{1}{|\tilde x-\tilde y|}.$$
Let $x_0\in\Omega$ be the maximal point of $\phi_\Omega$. Noting that $\|\phi_\Omega\|_{L^2(\Omega)}=1$, we have
$$\phi_\Omega(x_0)\ge |\Omega|^{-1/2}.$$
Choose a positive constant $0<\delta<R_0/2$ such that
$B(x_0,2\delta)\subset\Omega$
and
\begin{equation} \label{eq:in-ball-phi}
\inf_{x \in \bar{B}(x_0,\delta)} \phi_\Omega(x)\ge \phi_\Omega(x_0)-\delta \|\nabla \phi_\Omega\|_{L^\infty(\Omega)}>\frac 12 \phi_\Omega(x_0)>\frac 12 |\Omega|^{-1/2}.
\end{equation}
Besides, we have
$$\Gamma(\tilde x_0,\tilde y)
\le C\left(\frac{\rho_{\tilde\Omega}(\tilde x_0)\rho_{\tilde\Omega}(\tilde y)}{|\tilde x_0-\tilde y|^2}\wedge 1\right)\frac{1}{|\tilde x_0-\tilde y|}
\le C(\Omega,\delta),$$
for all $\tilde y\in \tilde\Omega \setminus B(\tilde x_0,\delta)$,
where $\tilde x_0:=(x_0,0)$.

Combining the above two estimates, one obtains
\begin{align*}
\inf_{\tilde x \in \bar{B}(\tilde x_0,\delta)}v(\tilde x)
& = \inf_{(x,t) \in \bar{B}(\tilde x_0,\delta)}
\phi_\Omega(x) e^{\sqrt\lambda t} \ge \frac 12 |\Omega|^{-1/2} e^{-\sqrt\lambda \delta} \\
&\ge C(\Omega,\delta)   \frac 1{C(\Omega,\delta)} |\Omega|^{-1/2} e^{-\sqrt\lambda \delta}\\
&\ge \frac 1{C(\Omega,\delta)} |\Omega|^{-1/2} e^{-\sqrt\lambda \delta}\Gamma(\tilde x_0,\tilde y),
\end{align*}
for any $\tilde y\in \tilde\Omega \setminus B(\tilde{x}_0,\delta)$. This together with the fact that
$$\Gamma(\tilde x_0,\tilde y)=0\le C(\Omega,\delta) |\Omega|^{1/2} e^{\sqrt\lambda \delta} v(\tilde y),
\quad \forall\,\tilde y\in \partial\tilde\Omega,$$
yields that
\[
\Gamma(\tilde x_0,\tilde y) \le C(\Omega,\delta) |\Omega|^{1/2} e^{\sqrt\lambda \delta} v(\tilde y),
  \quad \forall\,\tilde y\in \partial B(\tilde{x}_0,\delta) \cup \partial\tilde\Omega.
\]
Furthermore, by the maximal principle, we have
$$\Gamma(\tilde x_0,\tilde y)\le  C(\Omega,\delta) |\Omega|^{1/2} e^{\sqrt\lambda \delta} v(\tilde y),
 \quad \forall\, \tilde{y} \in \tilde\Omega \setminus B(\tilde x_0,\delta).$$
Recall that $R_0 = \diam(\Omega)$ and
$$\Omega\times (-2R_0,2R_0)\subset\tilde\Omega \subset \Omega\times (-4R_0,4R_0),$$
it holds that, for any $x \in \Omega$ and $\tilde x=(x,0)$,
$$\rho_{\tilde \Omega}(\tilde x)=\rho(x),$$
and hence for any $x \in \Omega \setminus B(x_0, \delta)$,
\begin{align*}
\phi_\Omega(x)
=v(\tilde x)
\ge C(\Omega,\delta)\Gamma(\tilde x_0,\tilde x)
& \ge C(\Omega,\delta) \left(\frac{\rho_{\tilde\Omega}(\tilde x_0)\rho_{\tilde\Omega}(\tilde x)}{|\tilde x_0-\tilde x|^2}\wedge 1\right)\frac{1}{|\tilde x_0-\tilde x|}  \\
& \ge C(\Omega,\delta) \left(\frac{\delta \rho_{\tilde\Omega}(\tilde x)}{R^2}\wedge \rho_{\tilde\Omega}(\tilde x)\right)\frac{1}{R} \\
& \ge C(\Omega,\delta) \rho_{\tilde\Omega}(\tilde x)
=C(\Omega,\delta) \rho(x),
\end{align*}
where we used the fact that $\rho_{\tilde\Omega}(\tilde x) \leq 2 R \leq C(\Omega)$. This together with \eqref{eq:in-ball-phi} yields that, for any $x \in \Omega$,
\[
\phi_{\Omega}(x) \ge C(\Omega,\delta) \rho(x),
\]
which completes the proof.
 \end{proof}

\subsubsection{Lower bound of the heat kernel} \hskip\parindent
With Theorem \ref{lower-eigenfunction} at hand, we can follow Zhang's approach \cite{zha02,zha03} to derive the lower bound for the heat kernel around the boundary.
By combining this boundary estimate with the work of Grigor'yan and Saloff-Coste  \cite{gsc02}, we establish the global lower bound.

We need the comparison result for parabolic equations from \cite[Theorem 1.6]{FGS86}.

\begin{thm}\label{comparison}
Let $\Omega \subset \mathbb{R}^n$ be a bounded Lipschitz domain, $x_0\in\partial \Omega$ and $t>0$, where $ n \ge 2$.
Assume that $\psi$ is a Lipschitz function satisfying $0<c\le \psi\le C<\infty$
, where $c$ and $C$ are positive constants. Suppose that $u,v$ are two positive solutions to the heat equation
$$\partial_t u=\Delta^\psi_{\Omega}u$$
on $\Omega \times (t/4,\infty)$, with both $u$ and $v$ vanishing continuously on
$B(x_0,16\sqrt{a_0t}) \cap\partial \Omega$, where $a_0 = 3/8$. Then there exists constants $r_0>0$ and $C>0$ depending
only on $\Omega$ and $\psi$, such that, for any $t\leq r_0$ and $x\in B(x_0,\sqrt{a_0t})$,
$$\frac{u(x,t)}{v(x,t)}\le C\frac{u(x',2t)}{v(x',t/2)},$$
where $x'\in\Omega$ satisfies $\mathrm{dist}(x',B(x_0,8\sqrt{a_0t})\cap\partial \Omega)=8\sqrt{a_0t}$.
\end{thm}

\begin{thm}\label{lower-bound-heat}
Let $\Omega\subset\rr^2$ be a connected exterior $C^{1,1}$ domain. Then there exists a constant $C> 0$
such that, for any $x, y \in \Omega$ and $t >0$,
 \begin{equation}\label{heat-plane-2}
p_\Omega^{\phi}(t,x,y)\ge  C\frac{1}{\sqrt{\mu(x,\sqrt t)\mu(y,\sqrt t)}}\left(\frac{\rho(x)}{\sqrt t\wedge 1}\wedge 1\right)
\left(\frac{\rho(y)}{\sqrt t\wedge 1}\wedge 1\right)e^{-\frac{|x-y|^2}{ct}}.
\end{equation}
\end{thm}

\begin{proof}
Let $R:= \diam(\Omega^c)$. Without loss of generality, we can assume $r_0 \leq R/24$, where $r_0$ is as in Theorem \ref{comparison}.
Fix a point $x_0 \in \Omega^c$ and define the truncated domain $\tilde\Omega:=B(x_0,4 R)\cap \Omega$, where $R=\mbox{diam}(\Omega^c)$.
Let $\psi_1$ and $\lambda_1$ denote the first eigenfunction and the corresponding eigenvalue of the Laplace operator $\Delta^\phi_{\tilde\Omega}$, respectively.

By \cite[Corollary 3.5]{gsc02} and \cite[Lemma 2.1]{zha03}, for any $t >0$ and $x,y \in \Omega$ satisfying
$\rho(x) \ge \sqrt{a_0 t} \wedge \sqrt{a_0 r_0}$ and $\rho(y) \ge \sqrt{a_0 t} \wedge
\sqrt{a_0 r_0}$, it holds that
\begin{align}\label{heat-plane-2a}
 p_\Omega^{\phi}(t, x,y)
 & \geq \frac{C}{\sqrt{\mu(x,\sqrt t)\mu(y,\sqrt t)}}e^{-\frac{|x-y|^2}{ct}}.
\end{align}

For the remaining part, we shall consider the small time part and large time part separately.

{\bf Case 1: \boldmath $0 < t \leq r_0$}. Assume that $\rho(x)< \sqrt{a_0t}$.
Let $\bar x\in\partial\Omega$ satisfy $\rho(x)=|x-\bar x|.$
We can choose a point $x'\in\Omega$ such that $x$, $\bar x$ and $x'$
are in the same geodesic and $\rho(x')=8 \sqrt{a_0t}=|x'-\bar x|.$

Now we write $v(x,t):=p_\Omega^{\phi}(t,x,y)$ and $u(x,t):=e^{-t\lambda_1}\psi_1(x)$.
Both $u$ and $v$ are positive solutions of
the heat equation $\partial_t u=\Delta^\phi_{\tilde\Omega} u$ in $\tilde{\Omega} \times (t/4, \infty)$.
Hence, it follows from Theorem \ref{comparison} that there exists a constant $C>0$ such that
\begin{align*}
\frac{u(x,t)}{v(x,t)}\le C\frac{u(x',2t)}{v(x',t/2)}.
\end{align*}
Namely, it holds that
\begin{align*}
p_\Omega^{\phi}(t,x,y)
&\ge C e^{t\lambda_1} \frac{\psi_1(x)}{\psi_1(x')} p^\phi_\Omega(t/2,x',y) \\
&\ge  C \frac{\rho_{\tilde{\Omega}}(x)}{\rho_{\tilde{\Omega}}(x')} p^\phi_\Omega(t/2,x',y),
\end{align*}
where the last inequality is due to Theorem \ref{lower-eigenfunction}.
Besides, by the fact $\rho(x) \vee \rho(x') < R$, one obtains that  $\rho_{\tilde{\Omega}}(x) = \rho(x)$ and $\rho_{\tilde{\Omega}}(x') =\rho(x') = 8 \sqrt{a_0 t}$,
which gives that
$$
p_\Omega^{\phi}(t,x,y)
\ge  C \frac{\rho(x)}{\sqrt{t}} p^\phi_\Omega(t/2,x',y).
$$

For the case $\rho(y) \ge \sqrt{a_0 t}$, it follows from \eqref{heat-plane-2a} that
\[
p^\phi_\Omega(t/2,x',y)
\ge C \frac{1}{\mu(B(y, \sqrt t))} e^{-\frac{|x'-y|^2}{ct}}
\ge C \frac{1}{\mu(B(y, \sqrt t))} e^{-\frac{|x-y|^2}{ct}},
\]
where we used the fact that $|x'-y| \leq |x-y| + |x-x'| \leq |x-y| + 8 \sqrt{a_0t}$ in the last inequality. Therefore, one has
\[
p^\phi_\Omega(t,x,y)
\ge  C \frac{\rho(x)}{\sqrt{t}}  \frac{1}{\sqrt{\mu(x,\sqrt t)\mu(y,\sqrt t)}}e^{-\frac{|x-y|^2}{ct}}.
\]

For the case $\rho(y)< \sqrt{a_0 t}$, let $\bar{y} \in \partial \Omega$
satisfy $\rho(y) = |y - \bar{y}|$.
We can choose $y' \in \Omega$ such that $y$, $\bar y$ and $y'$
are in the same geodesic and $\rho(y')=8\sqrt{a_0t}$.
Using Theorem \ref{comparison} and \eqref{heat-plane-2a},
one sees that
\begin{align*}
  p^\phi_\Omega(t/2,x',y)
  &\ge C e^{\frac{1}{2}\lambda_1 t} \frac{\rho(y)}{\sqrt{t}} p^\phi_\Omega(t/4,x',y')  \\
  &\ge C \frac{\rho(y)}{\sqrt{t}} \frac{1}{\mu(y',\sqrt{t})} e^{-\frac{|x'-y'|^2}{ct}} \\
  &\ge C \frac{\rho(y)}{\sqrt{t}} \frac{1}{\mu(x,\sqrt{t})} e^{-\frac{|x-y'|^2}{ct}} \\
  &\ge C \frac{\rho(y)}{\sqrt{t}} \frac{1}{\mu(x,\sqrt{t})} e^{-\frac{|x-y|^2}{ct}},
\end{align*}
where we used the facts that
\[
|y'-x'| \leq |x-x'| + |x-y'| \leq |x-y'| + 8 \sqrt{a_0 t}
\]
and
\[
|x-y'| \leq |x-y| + |y-y'| \leq |x-y| + 8 \sqrt{a_0 t}.
\]
Therefore, it holds that
\begin{equation} \label{eq:sma-t}
p^\phi_\Omega(t/2,x',y)
\ge  C \frac{\rho(x)}{\sqrt{t}} \frac{\rho(y)}{\sqrt t} \frac{1}{\sqrt{\mu(x,\sqrt t)\mu(y,\sqrt t)}}e^{-\frac{|x-y|^2}{ct}}.
\end{equation}

{\bf Case 2: \boldmath $t \ge r_0$.} Assume that $\rho(x)< \sqrt{a_0 r_0}$.
Let $\bar x\in\partial\Omega$ satisfy $\rho(x)=|x-\bar x|.$
We can choose a point $x'\in\Omega$ such that $x$, $\bar x$ and $x'$
are in the same geodesic and $\rho(x')=8 \sqrt{a_0r_0}=|x'-\bar x|.$

Let $v(x,s):=p_\Omega^{\phi}(s+t_0,x,y)$ and
$u(x,s):=e^{-\lambda_1 s}\psi_1(x),$ where $t_0:=t-r_0$. Then
both $u$ and $v$ are positive solutions of
the heat equation $\partial_t u=\Delta^\phi_{\tilde\Omega} u$ in $\tilde{\Omega} \times (t/4, \infty)$.
By Theorem \ref{comparison} once more, we see that
\begin{align*}
v(x,r_0)&\ge C\frac{u(x,r_0) v(x',r_0/2)}{u(x',2r_0)}
\ge  Ce^{\lambda_1 r_0 }\frac{\psi_1(x)}{\psi_1(x')} p_\Omega^{\phi}(t_0+r_0/2,x',y).
\end{align*}
For the case $\rho(y)\ge \sqrt{a_0 r_0}$,
it follows from  \eqref{heat-plane-2a} that
$$
p_\Omega^{\phi}(t_0+r_0/2,x',y)
\ge C \frac{\rho(x)}{\sqrt r_0}
     \frac{1}{\mu(B(y,\sqrt{t_0 + r_0 /2}))}  e^{-\frac{|x'-y|^2}{c(t_0 + r_0 /2)}}
\ge C \frac{\rho(x)}{\sqrt{r_0}}\frac{1}{\sqrt{\mu(x,\sqrt t)\mu(y,\sqrt t)}}e^{-\frac{|x-y|^2}{ct}},
$$
where we used the facts that $ t/2 \le t_0 + r_0 /2 = t - r_0 /2 \le t$
and $|x'-y| \leq |x-y| + 8 \sqrt{a_0 r_0} \leq |x-y| + 8 \sqrt{a_0 t}$.
Therefore, it holds that
\begin{equation*}
p_{\Omega}^{\phi}(t,x,y)
\geq C \frac{\rho(x)}{\sqrt{r_0}}\frac{1}{\sqrt{\mu(x,\sqrt t)\mu(y,\sqrt t)}}e^{-\frac{|x-y|^2}{ct}}
\geq C \frac{\rho(x)}{\sqrt{t}}
\left(\frac{\rho(y)}{\sqrt t} \wedge 1 \right)
\frac{1}{\sqrt{\mu(x,\sqrt t)\mu(y,\sqrt t)}}e^{-\frac{|x-y|^2}{ct}}.
\end{equation*}
 For the case $\rho(y) < \sqrt{a_0 r_0}$, following the same argument as in \eqref{eq:sma-t}, one has
 \begin{equation*}
p_{\Omega}^{\phi}(t,x,y)
\geq C \frac{\rho(x)}{\sqrt{r_0}}\frac{\rho(y)}{\sqrt{r_0}}\frac{1}{\sqrt{\mu(x,\sqrt t)\mu(y,\sqrt t)}}e^{-\frac{|x-y|^2}{ct}}
\geq C \frac{\rho(x)}{\sqrt{t}}\frac{\rho(y)}{\sqrt{t}}\frac{1}{\sqrt{\mu(x,\sqrt t)\mu(y,\sqrt t)}}e^{-\frac{|x-y|^2}{ct}}.
\end{equation*}

Combining the estimates above for $p_{\Omega}^{\phi}(t,x,y)$, we find that the estimate \eqref{heat-plane-2} holds, which
completes the proof of Theorem \ref{lower-bound-heat}.
\end{proof}

From Theorems \ref{heat-upper} and \ref{lower-bound-heat}  together with Proposition
\ref{prop:harmonic-es},
we obtain the complete two-sided estimates for the heat kernel in the exterior domains on the plane.

\begin{thm}\label{twoside-bound-heat}
Let $\Omega\subset\rr^2$ be an exterior $C^{1,1}$ domain. Then it holds that
\begin{align*}
p_\Omega(t,x,y)
&\sim \frac{\log(e+\rho(x))\log(e+\rho(y))}{t\left(\log(e+\sqrt t)+\log(e+\rho(x))\right)^2}\left(\frac{\rho(x)}{\sqrt t\wedge 1}\wedge 1\right)\left(\frac{\rho(y)}{\sqrt t\wedge 1}\wedge 1\right)e^{-\frac{|x-y|^2}{ct}}\\
&\sim \frac{\log(e+\rho(x))\log(e+\rho(y))}{t\left(\log(e+\sqrt t)+\log(e+\rho(x)\vee\rho(y))\right)^2}\left(\frac{\rho(x)}{\sqrt t\wedge 1}\wedge 1\right)\left(\frac{\rho(y)}{\sqrt t\wedge 1}\wedge 1\right)e^{-\frac{|x-y|^2}{ct}}.
\end{align*}
Moreover, if $\Omega$ is an $\cdini$ domain, then the upper bound still holds.
\end{thm}

\section{Green function and Riesz potentials}\label{s3}\hskip\parindent
Let us use the previous heat kernel estimate to derive estimate for Green's function, $\Gamma_\Omega(x,y)$, which is defined by
\begin{align*}
\Gamma_\Omega(x,y):=\int_0^\infty p_\Omega(t,x,y)\,dt.
\end{align*}
Equivalently, $\Gamma_\Omega(x,y)$ is the infimum
of all positive fundamental solutions of
the Laplace operator $\Delta_{\Omega}$; see \cite{gsc02}.

\begin{thm}
Let $\Omega \subset \mathbb{R}^2$ be an exterior $C^{1,1}$ domain. Then, it holds that for any $x, y \in \Omega$,
$$
\Gamma_\Omega(x,y)\sim
\begin{dcases}
1 + \log \frac{\rho(x)\wedge \rho(y)}{|x-y|},  &\,  \mbox{if} \, |x-y|<\rho(x)\wedge \rho(y)\wedge 1, \\
\frac{\left(\rho(x)\wedge 1\right)\left(\rho(y)\wedge 1\right)}{|x-y|^2\wedge 1}\log (e+\rho(x)\wedge \rho(y)), &\, \mbox{otherwise}.
\end{dcases}
$$
\end{thm}
\begin{proof}
Similarly to  \cite[Theorem 5.2]{ly86} and \cite[Theorem 5.13]{gsa11}, it holds that
$$\Gamma_\Omega(x,y)\sim \int_{|x-y|^2}^\infty p_\Omega(t,x,y)\,dt.$$

By Theorem \ref{heat-upper}, for any $t \ge |x-y|^2$, one obtains
\begin{align*}
p_\Omega(t,x,y)
&\sim  \frac{\log(e+\rho(x))\log(e+\rho(y))}{t\left(\log(e+\sqrt t)+\log(e+\rho(y)\wedge \rho(x))\right)^2}
\left(\frac{\rho(x)}{\sqrt t\wedge 1}\wedge 1\right)\left( \frac{\rho(y)}{\sqrt t\wedge 1}\wedge 1\right) \\
& \sim  \frac{\log(e+\rho(x))\log(e+\rho(y))}{t\left(\log(e+\sqrt t)+\log(e+\rho(y)\vee\rho(x))\right)^2}
\left(\frac{\rho(x)}{\sqrt t\wedge 1}\wedge 1\right)\left( \frac{\rho(y)}{\sqrt t\wedge 1}\wedge 1\right).
\end{align*}
Next, we shall estimate $\Gamma_\Omega(x,y)$ by considering the following two cases.

{\bf Case 1: \boldmath $|x-y|\ge 1$.} In this case, one has
\begin{align*}
p_\Omega(t,x,y)&\sim   \frac{\log(e+\rho(x))\log(e+\rho(y))}{t\left(\log(e+\sqrt t)+\log(e+\rho(x) \wedge \rho(x))\right)^2}
\left(\rho(x)\wedge 1\right)\left(\rho(y)\wedge 1\right).
\end{align*}

{\bf Subcase 1.1: \boldmath $\rho(x)\wedge \rho(y)\le 1$.} In this subcase, we further have
\begin{align*}
p_\Omega(t,x,y)&\sim   \frac{\log(e+\rho(x))\log(e+\rho(y))}{t\log^2(e+\sqrt t)}\left(\rho(x)\wedge 1\right)\left(\rho(y)\wedge 1\right).
\end{align*}
%
Then, we see that
$$\rho(x)\vee\rho(y) \leq \rho(x) \wedge \rho(y) + |x-y| \leq 2|x-y|$$
and
$$|x-y| \leq \diam(\Omega^c) + \rho(x) + \rho(y) \lesssim e + \rho(x) \vee \rho(y).$$
It follows that
\begin{align*}
\int_{|x-y|^2}^\infty p_\Omega(t,x,y)\,dt
&\sim \int_{|x-y|^2}^\infty  \frac{\log(e+\rho(x)\wedge\rho(y))\log(e+\rho(x)\vee\rho(y))}{t\log^2(e+\sqrt t)}\left(\rho(x)\wedge 1\right)\left(\rho(y)\wedge 1\right)\,dt\nonumber\\
&\sim \left(\rho(x)\wedge 1\right)\left(\rho(y)\wedge 1\right) \frac{\log(e+\rho(x)\wedge \rho(y))\log(e+\rho(x) \vee \rho(y))}{\log(e+|x-y|)} \nonumber\\
&\sim \left(\rho(x)\wedge 1\right)\left(\rho(y)\wedge 1\right),
\end{align*}
since $e+\rho(x) \wedge \rho(y) \sim e$ and $e + \rho(x) \vee\rho(y) \sim e + |x-y|$.

{\bf Subcase 1.2: \boldmath $\rho(x)\wedge \rho(y)\ge 1$.} In this subcase, one has
\begin{align*}
p_\Omega(t,x,y)
&\sim  \frac{\log(e+\rho(x))\log(e+\rho(y))}{t\left(\log(e+\sqrt t)+\log(e+\rho(x)\vee \rho(y))\right)^2}
\end{align*}
and
$$
|x - y |
\leq \diam(\Omega^c) + \rho(x) + \rho(y) \lesssim \rho(x) \vee \rho(y).
$$
Thus, it follows that
\begin{align*}
\int_{|x-y|^2}^\infty p_\Omega(t,x,y)\,dt &\sim \int_{|x-y|^2}^\infty \frac{\log(e+\rho(x))\log(e+\rho(y))}{t\left(\log(e+\sqrt t)
+\log(e+\rho(x)\vee\rho(y))\right)^2}\,dt\\
&\sim \frac{\log(e + \rho(x) \wedge \rho(y)) \log(e + \rho(x) \vee \rho(y))}
{\log(e + |x-y|) + \log(e + \rho(x) \vee \rho(y))} \\
& \sim
\log(e+\rho(x)\wedge\rho(y)),
\end{align*}
since $\log(e + |x-y|) + \log(e + \rho(x) \vee \rho(y))
\sim \log(e + \rho(x) \vee \rho(y)) $.

{\bf Case 2: \boldmath $|x-y|<1$.} In this case, we shall further divide it into three subcases.

{\bf Subcase 2.1: \boldmath $\rho(x)\wedge \rho(y)\ge 1$.}
In this subcase, it holds that
\[
\left(\frac{\rho(x)}{\sqrt t\wedge 1}\wedge 1\right) = 1
\quad\text{and}\quad
\left(\frac{\rho(y)}{\sqrt t\wedge 1}\wedge 1\right) = 1,
\]
which means that
\begin{align*}
p_\Omega(t,x,y)&\sim  \frac{\log(e+\rho(x))\log(e+\rho(y))}{t\left(\log(e+\sqrt t)+\log(e+\rho(x)\vee\rho(y))\right)^2}.
\end{align*}
Besides, we have $\rho(x) \vee \rho(y) \leq |x-y| + \rho(x) \wedge \rho(y) \leq 2\rho(x) \wedge \rho(y)$ and hence $\rho(x) \vee \rho(y) \sim  \rho(x) \wedge \rho(y)$.
Then one obtains that
\begin{align*}
\int_{|x-y|^2}^\infty p_\Omega(t,x,y)\,dt
&\sim \int_{|x-y|^2}^{(\rho(x) \wedge \rho(y))^2} \frac{\log(e+\rho(x))\log(e+\rho(y))}{t\left(\log(e+\sqrt t)+\log(e+\rho(x)\vee\rho(y))\right)^2}\,dt\\
&\quad+ \int_{(\rho(x) \wedge \rho(y))^2}^\infty \frac{\log(e+\rho(x))\log(e+\rho(y))}{t\left(\log(e+\sqrt t)+\log(e+\rho(x)\vee\rho(y))\right)^2}\,dt\\
& \sim \frac{\log(e+\rho(x))\log(e+\rho(y))}{(\log(e+\rho(x)\vee\rho(y)))^2}
  \int_{|x-y|^2}^{(\rho(x) \wedge \rho(y))^2}  \frac{1}{t} \, dt \\
&\quad+ \frac{\log(e+\rho(x))\log(e+\rho(y))}{\log(e+ \rho(x)\wedge \rho(y)))+\log(e+\rho(x)\vee\rho(y))} \\
&\sim \log \frac{\rho(x)\wedge \rho(y)}{|x-y|}+\log(e+\rho(x)\wedge \rho(y)).
\end{align*}
If $\rho(x) \wedge \rho(y) \ge 2 |x-y|$, then it holds that
\[
\log \frac{\rho(x)\wedge \rho(y)}{|x-y|}+\log(e+\rho(x)\wedge \rho(y)) \sim \log \frac{\rho(x)\wedge \rho(y)}{|x-y|}
\sim  1  + \log \frac{\rho(x)\wedge \rho(y)}{|x-y|}.
\]
If $|x-y| < \rho(x) \wedge \rho(y) < 2 |x-y| < 2$, then one has
\[
\log \frac{\rho(x)\wedge \rho(y)}{|x-y|}+\log(e+\rho(x)\wedge \rho(y)) \sim 1
+ \log \frac{\rho(x)\wedge \rho(y)}{|x-y|}.
\]
Thus, in this subcase, we have
$$
\int_{|x-y|^2}^\infty p_\Omega(t,x,y)\,dt
\sim 1
+ \log \frac{\rho(x)\wedge \rho(y)}{|x-y|}.
$$

{\bf Subcase 2.2: \boldmath $\rho(x)\wedge \rho(y)<1\le \rho(x)\vee \rho(y)$.} In this subcase, it holds that
\begin{align*}
p_\Omega(t,x,y)&\sim  \frac{\log(e+\rho(x))\log(e+\rho(y))}{t\left(\log(e+\sqrt t)+\log(e+\rho(x)\vee\rho(y))\right)^2}
\left(\frac{\rho(x)\wedge \rho(y)}{\sqrt t\wedge 1}\wedge 1\right).
\end{align*}
If $|x-y|\le\rho(x)\wedge \rho(y) $, then one has $\rho(x)\wedge \rho(y)\sim \rho(x)\vee \rho(y)$, which means that for any $0 <t \leq 1$,
$$
\frac{\log(e+\rho(x))\log(e+\rho(y))}{\left(\log(e+\sqrt t)+\log(e+\rho(x)\vee\rho(y))\right)^2}
\sim
\frac{\log(e+\rho(x))\log(e+\rho(y))}{(\log(e+\rho(x)\vee\rho(y)))^2}
\sim 1.$$
This further yields that
\begin{align*}
\int_{|x-y|^2}^\infty p_\Omega(t,x,y)\,dt
&\sim \int_{|x-y|^2}^{(\rho(x)\wedge \rho(y))^2} \frac 1t \,dt+  \int_{(\rho(x)\wedge \rho(y))^2}^1 \frac {\rho(x)\wedge \rho(y)}{t^{3/2}} \,dt\nonumber\\
&\ +\int_{1}^\infty  \frac{\log(e+\rho(x))\log(e+\rho(y))}{t\left(\log(e+\sqrt t)+\log(e+\rho(x)\vee\rho(y))\right)^2}(\rho(x)\wedge \rho(y))\,dt\nonumber \\
&\sim  \log\frac{(\rho(x)\wedge \rho(y))}{|x-y|}
 + 1 - \rho(x) \wedge \rho(y) +
 \rho(x) \wedge \rho(y)
 \frac{\log(e+\rho(x))\log(e+\rho(y))}{\log(e + \rho(x) \vee \rho(y))} \\
&\sim  \log\frac{(\rho(x)\wedge \rho(y))}{|x-y|}
 + 1 - \rho(x) \wedge \rho(y) + \rho(x) \wedge \rho(y) \log(e + \rho(x) \wedge \rho(y) ) \\
&\sim \log\frac{(\rho(x)\wedge \rho(y))}{|x-y|} + 1 - \rho(x) \wedge \rho(y) +  \rho(x) \wedge \rho(y) \\
& \sim 1 +  \log\frac{(\rho(x)\wedge \rho(y))}{|x-y|}.
\end{align*}

If $|x-y|>\rho(x)\wedge \rho(y)$, then we have
$1 \leq \rho(x)\vee \rho(y) \leq \rho(x) \wedge \rho(y) + |x-y|
\leq  2 |x-y|,$
which implies that for any $t \ge |x-y|^2 \ge 1/4$,
\[
\left(\frac{\rho(x)\wedge \rho(y)}{\sqrt t\wedge 1}\wedge 1\right)
\sim
\rho(x) \wedge \rho(y).
\]
It follows that
\begin{align*}
\int_{|x-y|^2}^\infty p_\Omega(t,x,y)\,dt &\sim \int_{|x-y|^2}^\infty  \frac{\log(e+\rho(x))\log(e+\rho(y))}{t\left(\log(e+\sqrt t)+\log(e+\rho(x)\vee\rho(y))\right)^2}(\rho(x)\wedge \rho(y))\,dt\nonumber \\
&\sim  \rho(x) \wedge \rho(y)
  \frac{\log(e+\rho(x))\log(e+\rho(y))}{\log(e+ |x-y|) + \log(e + \rho(x) \vee \rho(y))} \\
& \sim \rho(x) \wedge \rho(y)  \frac{\log(e+\rho(x))\log(e+\rho(y))}{\log(e + \rho(x) \vee \rho(y))}  \\
& \sim \rho(x) \wedge \rho(y)  \log(e+\rho(x) \wedge \rho(y)) \\
& \sim \rho(x) \wedge \rho(y)
 \sim \left(\rho(x)\wedge 1\right)\left(\rho(y)\wedge 1\right),
\end{align*}
since $\rho(x) \wedge \rho(y) < 1 \leq \rho(x) \vee \rho(y)$.

{\bf Subcase 2.3: \boldmath $\rho(x)\wedge \rho(y)\le \rho(x)\vee \rho(y)<1$.}
In this subcase, we have for any  $0 <  t \leq  1$,
\[
\frac{\log(e+\rho(x))\log(e+\rho(y))}{\left(\log(e+\sqrt t)
+\log(e+\rho(x)\vee\rho(y))\right)^2}\sim 1.
\]
Besides, for any $t \ge 1$, it holds that
\[
    \frac{\log(e+\rho(x))\log(e+\rho(y))}{\left(\log(e+\sqrt t)
    +\log(e+\rho(x)\vee\rho(y))\right)^2}
    \sim \frac{1}{ (\log(e + \sqrt t))^2}.
\]

If $|x-y|< \rho(x)\wedge \rho(y)$,  then we have
\begin{align*}
\int_{|x-y|^2}^\infty p_\Omega(t,x,y)\,dt
& \sim \int_{|x-y|^2}^{(\rho(x)\wedge \rho(y))^2} p_\Omega(t,x,y)\,dt
 + \int_{(\rho(x)\wedge \rho(y))^2}^{(\rho(x)\vee \rho(y))^2}   p_\Omega(t,x,y)\,dt \\
&\quad+ \int_{(\rho(x)\vee \rho(y))^2}^1   p_\Omega(t,x,y)\,dt
+\int_1^\infty   p_\Omega(t,x,y)\,dt \\
&\sim \int_{|x-y|^2}^{(\rho(x)\wedge \rho(y))^2} \frac{1}{t} \,dt
  + \int_{(\rho(x)\wedge \rho(y))^2}^{(\rho(x)\vee \rho(y))^2}
  \frac{1}{t} \frac{\rho(x) \wedge \rho(y)}{\sqrt t}\,dt \\
&\quad+ \int_{(\rho(x)\vee \rho(y))^2}^1  \frac{1}{t} \frac{\rho(x) \wedge \rho(y)}{\sqrt t}  \frac{\rho(x) \vee \rho(y)}{\sqrt t}\,dt
 + \int_1^\infty   \frac{\rho(x) \rho(y)}{t (\log(e + \sqrt t))^2} \,dt \\
& \sim \log \frac{\rho(x)\wedge \rho(y)}{|x-y|}
+  \rho(x) \wedge \rho(y) \left( \frac{1}{\rho(x) \wedge \rho(y)} - \frac{1}{\rho(x) \vee \rho(y)} \right) \\
&\quad+ \rho(x) \wedge \rho(y) \rho(x) \vee \rho(y)
   \left( \frac{1}{ (\rho(x) \vee \rho(y))^2}  - 1 \right)
+ \rho(x) \rho(y)  \\
& \sim \log \frac{\rho(x)\wedge \rho(y)}{|x-y|}
 + 1- \frac{\rho(x) \wedge \rho(y)}{\rho(x) \vee \rho(y)}
 + \frac{\rho(x) \wedge \rho(y)}{\rho(x) \vee \rho(y)} - \rho(x) \rho(y)
  + \rho(x) \rho(y)  \\
 & \sim 1 + \log \frac{\rho(x)\wedge \rho(y)}{|x-y|}.
\end{align*}

If $\rho(x)\wedge \rho(y)\le |x-y|$, then we have
$$\rho(x)\vee \rho(y)\le |x-y|+\rho(x)\wedge \rho(y)\le 2|x-y|,$$
which further yields that
\begin{align*}
\int_{|x-y|^2}^\infty p_\Omega(t,x,y)\,dt
& \sim \int_{|x-y|^2}^1   p_\Omega(t,x,y)\,dt
        + \int_1^\infty   p_\Omega(t,x,y)\,dt \\
& \sim \int_{|x-y|^2}^1  \frac{1}{t} \frac{\rho(x)}{\sqrt t} \frac{\rho(y)}{\sqrt t} \,dt + \int_1^\infty
\frac{\rho(x) \rho(y)}{t(\log(e + \sqrt t))^2}\,dt \\
& \sim  \frac{\rho(x)\rho(y)}{|x-y|^2}-\rho(x)\rho(y) + \rho(x) \rho(y)
 \sim \frac{\rho(x)\rho(y)}{|x-y|^2}.
\end{align*}

Summarizing  the above estimates, we get the desired estimate.
\end{proof}

\begin{rem}\rm
If we assume that $\Omega$ is an exterior $\cdini$ domain, then
the upper bound for $\Gamma_\Omega$ still holds.
\end{rem}

\begin{prop}\label{riesz-potential-1}
Let $\Omega$ be an exterior $\cdini$ domain in the plane. Then, for any $0<s<2$, the Riesz potential
$$
(-\Delta_{\Omega})^{-\frac{s}{2}}(x, y):=\frac{1}{\Gamma(s/2)} \int_0^{\infty} t^{\frac{s}{2}} e^{t \Delta_{\Omega}}(x, y)\,\frac{dt}{t}
$$
satisfies
$$(-\Delta_{\Omega})^{-\frac{s}{2}}(x, y)\lesssim \frac{1}{|x-y|^{2-s}}
\frac{\log(e+\rho(x)\wedge \rho(y))}{\log(e+\rho(y)\vee\rho(x))}\left(\frac{\rho(x)}
{|x-y| \wedge 1} \wedge 1\right)
\left(\frac{\rho(y)}{|x-y| \wedge 1} \wedge 1\right).$$
\end{prop}
\begin{proof}
It follows from Theorem \ref{heat-upper} that
\begin{align*}
p_\Omega(t,x,y)&\lesssim  \frac{\log(e+\rho(x))\log(e+\rho(y))}{t\left(\log(e+\sqrt t)+\log(e+\rho(y)\vee\rho(x))\right)^2}
\left(\frac{\rho(x)}{\sqrt t\wedge 1}\wedge 1\right)\left( \frac{\rho(y)}{\sqrt t\wedge 1}\wedge 1\right)
e^{-\frac{|x-y|^2}{ct}} \\
& \lesssim \frac{\log(e+\rho(x)\wedge \rho(y))}{\log(e+\rho(y)\vee\rho(x))}
\frac{1}{t} \left(\frac{\rho(x)}{\sqrt t\wedge 1}\wedge 1\right)\left( \frac{\rho(y)}{\sqrt t\wedge 1}\wedge 1\right)
e^{-\frac{|x-y|^2}{ct}}
\end{align*}
and hence
\[
 (-\Delta_{\Omega})^{-\frac{s}{2}}(x, y)\lesssim
\frac{\log(e+\rho(x)\wedge \rho(y))}{\log(e+\rho(y)\vee\rho(x))}
\int_0^\infty J(x,y,t) dt
= \frac{\log(e+\rho(x)\wedge \rho(y))}{\log(e+\rho(y)\vee\rho(x))}
J(x,y),
\]
where
\[
   J(x,y,t):= t^{\frac{s}{2}-2}\left(\frac{\rho(x)}{\sqrt t\wedge 1}\wedge 1\right)\left( \frac{\rho(y)}{\sqrt t\wedge 1}\wedge 1\right)
e^{-\frac{|x-y|^2}{ct}}.
\]
Next, we shall consider the following three cases separately.

{\bf Case 1:  \boldmath $\rho(x)\wedge \rho(y) \ge 1$.} In this case, it holds that
\[
    \left(\frac{\rho(x)}{ |x-y|\wedge 1}\wedge 1\right) = 1,
    \quad
     \left(\frac{\rho(y)}{ |x-y|\wedge 1}\wedge 1\right) = 1,
\]
and
\[
 J(x,y,t) = t^{\frac{s}{2}-2}e^{-\frac{|x-y|^2}{ct}}.
\]
It follows that
\begin{align*}
    J(x,y)
    &=   \int_0^{|x-y|^2} J(x,y,t) \,dt
    + \int_{|x-y|^2}^\infty J(x,y,t)  \,dt
     \\
    & \lesssim \int_0^{|x-y|^2} t^{\frac{s}{2}-2} \frac{t^{2-\frac{s}{2}}}{|x-y|^{4-s}} \,dt
     +  \int_{|x-y|^2}^\infty t^{\frac{s}{2}-2} \,dt \\
    &\sim \frac{1}{|x-y|^{2-s}}.
\end{align*}

{\bf Case 2: \boldmath $\rho(x)\wedge \rho(y) <1\le \rho(x)\vee \rho(y) $.} In this case, it holds that
$$
\left(\frac{\rho(x)}
{|x-y| \wedge 1} \wedge 1\right)
\left(\frac{\rho(y)}{|x-y| \wedge 1} \wedge 1\right)
= \left(\frac{\rho(x)\wedge \rho(y)}{ |x-y|\wedge 1}\wedge 1\right)
$$
and
$$
J(x,y,t)
  =t^{\frac{s}{2}-2} e^{-\frac{|x-y|^2}{ct}}\left(\frac{\rho(x)\wedge \rho(y)}{\sqrt t\wedge 1}\wedge 1\right).
$$

If $|x-y|< \rho(x)\wedge \rho(y)$, then one has
\[
    \frac{\rho(x) \wedge \rho(y)}{|x-y| \wedge 1} \wedge 1
     = \frac{\rho(x) \wedge \rho(y)}{|x-y|} \wedge 1 = 1.
\]
Besides, it holds that
\[
J(x,y,t)
\leq  t^{\frac{s}{2}-2} e^{-\frac{|x-y|^2}{ct}}
\]
and
\begin{align*}
  J(x,y)
& \lesssim \int_0^{|x-y|^2} t^{\frac{s}{2}-2}
e^{-\frac{|x-y|^2}{ct}}
\,dt
 +  \int_{|x-y|^2}^\infty t^{\frac{s}{2}-2} \,dt
\lesssim  \frac{1}{|x-y|^{2-s}}.
\end{align*}

If $|x-y|\ge \rho(x)\wedge \rho(y)$, then one has
$$1\leq \rho(x)\vee \rho(y)\leq |x-y|
+ \rho(x) \wedge \rho(y) \leq 2|x-y|$$
and
$$
  \frac{\rho(x) \wedge \rho(y)}{|x-y| \wedge 1}
    \sim \rho(x) \wedge \rho(y).
$$
It follows that
\begin{align*}
  J(x,y)
  & = \int_0^{(\rho(x)\wedge \rho(y))^2} J(x,y,t)\,dt
  + \int_{(\rho(x)\wedge \rho(y))^2}^{1}J(x,y,t) \,dt \\
  &\quad + \int_1^{4|x-y|^2} J(x,y,t) \,dt
   + \int_{4|x-y|^2}^\infty J(x,y,t) \,dt \\
  & \lesssim \int_0^{(\rho(x)\wedge \rho(y))^2} t^{\frac{s}{2}-2}
e^{-\frac{|x-y|^2}{ct}} \,dt
   + \int_{(\rho(x)\wedge \rho(y))^2}^{1} t^{\frac{s}{2}-2} \frac{\rho(x) \wedge \rho(y)}{\sqrt t} e^{-\frac{|x-y|^2}{ct}} \,dt  \\
  &\quad + \int_1^{4|x-y|^2}  t^{\frac{s}{2}-2} \rho(x) \wedge \rho(y) e^{-\frac{|x-y|^2}{ct}} \,dt
   + \int_{4|x-y|^2}^\infty t^{\frac{s}{2}-2} \rho(x) \wedge \rho(y)  \,dt \\
  & \lesssim \frac{(\rho(x) \wedge \rho(y))^2}{|x-y|^{4-s}}
     + \frac{\rho(x) \wedge \rho(y)}{|x-y|^{5-s}}
     + \frac{\rho(x) \wedge \rho(y)}{|x-y|^{2-s}}
     + \frac{\rho(x) \wedge \rho(y)}{|x-y|^{2-s}} \\
  & \lesssim \frac{\rho(x) \wedge \rho(y)}{|x-y|^{2-s}}.
\end{align*}

{\bf Case 3: \boldmath $\rho(x)\vee \rho(y)<1$.}
If $|x-y|< \rho(x)\wedge \rho(y)$, then it holds that
$$
\left(\frac{\rho(x)}
{|x-y| \wedge 1} \wedge 1\right)
\left(\frac{\rho(y)}{|x-y| \wedge 1} \wedge 1\right)
= 1.
$$
It follows that
\begin{align*}
  J(x,y)
& \lesssim \int_0^{|x-y|^2} t^{\frac{s}{2}-2}
e^{-\frac{|x-y|^2}{ct}}
\,dt
 +  \int_{|x-y|^2}^\infty t^{\frac{s}{2}-2} \,dt
\lesssim  \frac{1}{|x-y|^{2-s}}.
\end{align*}

If $|x-y|\ge  \rho(x)\wedge \rho(y)$, then one has
$$\rho(x)\vee\rho(y) \leq 2 |x-y|.$$
For the case $|x-y| < \frac{1}{2}$, it holds that
$$
\left(\frac{\rho(x)}
{|x-y| \wedge 1} \wedge 1\right)
\left(\frac{\rho(y)}{|x-y| \wedge 1} \wedge 1\right)
\sim \frac{\rho(x) \rho(y)}{|x-y|^2}.
$$
One can further conclude that
\begin{align*}
  J(x,y)
  & = \int_0^{(\rho(x)\wedge \rho(y))^2} J(x,y,t)\,dt
  + \int_{(\rho(x)\wedge \rho(y))^2}^{(\rho(x)\vee \rho(y))^2} J(x,y,t) \,dt \\
  &\quad + \int_{(\rho(x)\vee \rho(y))^2}^{4|x-y|^2} J(x,y,t) \,dt
   + \int_{4|x-y|^2}^1 J(x,y,t) \,dt
   + \int_1^\infty J(x,y,t) \,dt \\
   &\lesssim  \int_0^{(\rho(x)\wedge \rho(y))^2} t^{\frac{s}{2}-2} e^{-\frac{|x-y|^2}{ct}} \,dt
  + \int_{(\rho(x)\wedge \rho(y))^2}^{(\rho(x)\vee \rho(y))^2} t^{\frac{s}{2}-2}
  \frac{\rho(x) \wedge \rho(y)}{\sqrt t} e^{-\frac{|x-y|^2}{ct}} \,dt \\
&\quad + \int_{(\rho(x)\vee \rho(y))^2}^{4|x-y|^2} t^{\frac{s}{2}-2}  \frac{\rho(x)}{\sqrt t} \frac{\rho(y)}{\sqrt t}  e^{-\frac{|x-y|^2}{ct}} \,dt
 + \int_{4|x-y|^2}^1 t^{\frac{s}{2}-2}  \frac{\rho(x)}{\sqrt t} \frac{\rho(y)}{\sqrt t} \,dt  + \int_1^\infty  t^{\frac{s}{2}-2} \rho(x) \rho(y)\,dt  \\
 & \lesssim \frac{(\rho(x)\wedge \rho(y))^2}{|x-y|^{4-s}}
              + \frac{\rho(x) \wedge \rho(y)(\rho(x)\vee \rho(y))^2}{|x-y|^{5-s}}
    + \frac{\rho(x)\rho(y)}{|x-y|^{4-s}} + \frac{\rho(x)\rho(y)}{|x-y|^{2-s}} + \rho(x)\rho(y) \\
 & \lesssim \frac{\rho(x)\rho(y)}{|x-y|^{4-s}}.
\end{align*}
For the case $|x-y| \ge \frac{1}{2}$, it holds that
$$
\left(\frac{\rho(x)}
{|x-y| \wedge 1} \wedge 1\right)
\left(\frac{\rho(y)}{|x-y| \wedge 1} \wedge 1\right)
\sim \rho(x) \rho(y).
$$
One hence concludes that
\begin{align*}
J(x,y)&= \int_0^{(\rho(x)\wedge \rho(y))^2} J(x,y,t)\,dt
  + \int_{(\rho(x)\wedge \rho(y))^2}^{(\rho(x)\vee \rho(y))^2} J(x,y,t) \,dt \\
  & \quad+ \int_{(\rho(x)\vee \rho(y))^2}^{1} J(x,y,t) \,dt
      + \int_1^{4|x-y|^2} J(x,y,t) \,dt
      + \int_{4|x-y|^2}^\infty J(x,y,t) \,dt  \\
  & \lesssim \int_0^{(\rho(x)\wedge \rho(y))^2} t^{\frac{s}{2}-2} e^{-\frac{|x-y|^2}{ct}} \,dt
  + \int_{(\rho(x)\wedge \rho(y))^2}^{(\rho(x)\vee \rho(y))^2} t^{\frac{s}{2}-2}
  \frac{\rho(x) \wedge \rho(y)}{\sqrt t} e^{-\frac{|x-y|^2}{ct}} \,dt \\
  & \quad+ \int_{(\rho(x)\vee \rho(y))^2}^{1} t^{\frac{s}{2}-2}  \frac{\rho(x)}{\sqrt t} \frac{\rho(y)}{\sqrt t}  e^{-\frac{|x-y|^2}{ct}} \,dt
  + \int_1^{4|x-y|^2} t^{\frac{s}{2}-2} \rho(x) \rho(y)  e^{-\frac{|x-y|^2}{ct}} \,dt\\
  & \quad+ \int_{4|x-y|^2}^{\infty} t^{\frac{s}{2}-2} \rho(x) \rho(y)  \,dt \\
  &\lesssim \frac{(\rho(x)\wedge \rho(y))^2}{|x-y|^{4-s}} +
             \frac{\rho(x)\wedge \rho(y)(\rho(x)\vee \rho(y))^2}{|x-y|^{5-s}}
     + \frac{\rho(x)\rho(y)}{|x-y|^{6-s}}  + \frac{\rho(x) \rho(y)}{|x-y|^{2-s}} + \frac{\rho(x) \rho(y)}{|x-y|^{2-s}}  \\
 &\lesssim \frac{\rho(x) \rho(y)}{|x-y|^{2-s}}.
\end{align*}

Collecting the estimates above, we finish the proof of the proposition.
\end{proof}

Using the upper heat kernel estimate in higher dimensions (see Remark \ref{rem-heat-n3}),
the same argument as in Proposition \ref{riesz-potential-1} gives that

\begin{prop} \label{lem:riesz-poten}
Let $\Omega$ be an exterior $\cdini$ domain in $\rn$, $n \ge 3$.
 Suppose $0 < s < n$. Then the Riesz potential
$$
(-\Delta_{\Omega})^{-\frac{s}{2}}(x, y)
:=\frac{1}{\Gamma(s / 2)} \int_0^{\infty} t^{\frac{s}{2}} e^{t \Delta_{\Omega}}(x, y) \frac{d t}{t}
$$
satisfies
$$
(-\Delta_{\Omega})^{-\frac{s}{2}}(x, y)\lesssim \frac{1}{|x-y|^{n-s}}
\left(\frac{\rho(x)}
{|x-y| \wedge 1} \wedge 1\right)
\left(\frac{\rho(y)}{|x-y| \wedge 1} \wedge 1\right),
$$
uniformly for $x, y \in \Omega$.
\end{prop}

\section{Fractional Riesz transforms}\label{s4}\hskip\parindent
As the heat kernel estimates on an exterior $\cdini$ domain has been established,
we can follow the approach in \cite{KVZ16} to prove the boundedness of
the fractional Riesz transforms.
By employing the estimates for Riesz potentials (see Propositions \ref{riesz-potential-1} and \ref{lem:riesz-poten})
and applying Schur's test to determine the boundedness of integral operators with sign-definite kernels,
Hardy's inequality is established for exterior domains concerning both the Euclidean Laplacian
and the Dirichlet Laplacian. Next, by using the half-space as a comparison domain, we obtain an estimate
for the difference between the heat kernels in the Euclidean space and in domains.
Finally, Theorem \ref{main-1} is proved through Littlewood-Paley equivalence and
an estimate of difference between the square functions.

\subsection{Hardy inequality}

\begin{lem} \label{lem:schur-test}
(Schur's test with weights).
Assume that $(X, d \mu)$ and $(Y, d \nu)$ are measure spaces and $w(x, y)$ is a positive measurable function
defined on $X \times Y$. Let $K(x, y):\ X \times Y\rightarrow \mathbb{C}$ satisfy
\begin{equation}\label{eq:assump1-in-Schur}
\sup _{x \in X} \int_Y w(x, y)^{\frac{1}{p}}|K(x, y)|\, d\nu(y)=C_0<\infty
\end{equation}
and
\begin{equation}\label{eq:assump2-in-Schur}
\sup _{y \in Y} \int_X w(x, y)^{-\frac{1}{p'}}|K(x, y)|\, d\mu(x)=C_1<\infty
\end{equation}
for some $1<p<\infty$. Then the operator defined by
$$
Tf(x)=\int_Y K(x, y) f(y)\,d\nu(y)
$$
is a bounded operator from $L^p(Y, d\nu)$ to $L^p(X, d \mu)$. In particular,
$$
\|T f\|_{L^p(X, d \mu)} \lesssim C_0^{\frac{1}{p}} C_1^{\frac{1}{p}}
\|f\|_{L^p(Y, d\nu)}
$$
with the implicit positive constant independent of $f$.
\end{lem}
\begin{proof}
  For the proof, we refer to \cite[Lemma 5.1]{KVZ16}.
\end{proof}

\begin{prop} \label{prop:HardyForWhole} (Hardy inequality for $\Delta $).
Let $\Omega \subset \mathbb{R}^n$ be an exterior Lipschitz domain,
$n \geq 2$. Assume $1<p<\infty$ and $0<s<n/p$.
Then there exists a positive constant $C$ such that, for any $f \in C_c^{\infty}(\Omega)$,
\begin{equation*}
\left\|\frac{f}{\rho^s}\right\|_{L^p(\Omega)}\le C\left\|(-\Delta )^{s/2}f\right\|_{L^p(\mathbb{R}^n)}.
\end{equation*}
\end{prop}
\begin{proof}
By translation, we may assume that $0 \in \Omega^c$. We will use the abbreviation $\abdiam:=\diam(\Omega^c)$.
Let $\phi\in C_c^{\infty}(\mathbb{R}^n)$ be a smooth bump function
such that $\phi(x)= 1$ if $x \in B(0, 2\abdiam)$ and $\phi(x) = 0$ if
$x \in B(0, 3\abdiam)^c$. Notice that, for $x \in \supp(1-\phi)$,
$\rho(x) \sim |x|$.

We decompose $f= \phi f + (1-\phi) f$ and deal with the two parts separately.
We now estimate the first part. Observe that $\phi f$ is supported in the
bounded Lipschitz domain $U:=\Omega \cap B(0,3\abdiam)$.
Therefore, by the Hardy inequality for such domains (cf. \cite[Proposition 5.7]{Tr01})
\footnote{Although \cite[Proposition 5.7]{Tr01} is stated for a bounded smooth domain, a detailed investigation of
its proof reveals that a bounded Lipschitz domain suffices for $s>0$.}, we obtain
\begin{equation}\label{eq:fir-inHardy-whole}
\left\|\frac{\phi(x) f(x)}{\rho(x)^s}\right\|_{L^p(\Omega)} \lesssim\left\|\frac{\phi(x) f(x)}{\operatorname{dist}(x, \partial U)^s}\right\|_{L^p(U)} \lesssim\left\|\left(-\Delta \right)^{s / 2}(\phi f)\right\|_{L^p\left(\mathbb{R}^n\right)}.
\end{equation}
When $0<s<1$, applying the fractional product rule (cf. \cite{KVZ16}) and the Sobolev embedding theorem,
we have
\begin{align}\label{eq:4.1}
&\left\|(-\Delta )^{s/2}(\phi f)\right\|_{L^p(\mathbb{R}^n)} \\ \nonumber
&\ \lesssim\|\phi\|_{L^{\infty}(\mathbb{R}^n)}
\left\|(-\Delta )^{s/2} f\right\|_{L^p(\mathbb{R}^n)}
+\left\|(-\Delta )^{s/2} \phi\right\|_{L^{\frac{n}{s}}(\mathrm{R}^n)}
\|f\|_{L^{\frac{p n}{n-s p}}(\mathbb{R}^n)} \\ \nonumber
& \ \lesssim\left\|(-\Delta )^{s/2} f\right\|_{L^p(\mathbb{R}^n)}.
\end{align}
Combining this with \eqref{eq:fir-inHardy-whole} gives the desired bound
for $\phi f$ in the case $0<s<1$. When $s \geq 1$, let $s=k+\epsilon$ with
$k$ being an integer and $0 \leq \epsilon<1$. Then, using the boundedness of
the Riesz transform on $L^p(\mathbb{R}^n)$ and the regular (pointwise) product rule,
we have
$$
\left\|(-\Delta )^{s/2}(\phi f)\right\|_{L^p(\mathbb{R}^n)}
\lesssim \sum_{|\alpha|+|\beta|=k}
\left\|(-\Delta )^{\epsilon/2} [(\partial^\beta \phi)(\partial^\alpha f)]\right\|_{L^p(\mathbb{R}^n)}.
$$
This together with the fractional product rule, the Sobolev embedding theorem,
and the H\"older inequality in much the same manner as before further yields that
\begin{equation}\label{eq:4.2}
\left\|(-\Delta )^{s/2}(\phi f)\right\|_{L^p(\mathbb{R}^n)}
\lesssim
\left\|(-\Delta )^{s/2} f\right\|_{L^p(\mathbb{R}^n)}.
\end{equation}
Combining this with \eqref{eq:fir-inHardy-whole} provides the desired estimate
on $\phi f$ in the case $s\ge1$.

To estimate $[1-\phi] f$, we use the classical Hardy inequality on $\mathbb{R}^n$;
this requires $s<\frac{n}{p}$. Noting that $\rho(x) \sim |x|$ for any $x\in\supp(1-\phi)$, we have
\begin{align*}
\left\|\frac{[1-\phi(x)] f(x)}{\rho(x)^s}\right\|_{L^p(\Omega)}
& \lesssim \left\|\frac{[1-\phi(x)] f(x)}{|x|^s}\right\|_{L^p(\mathbb{R}^n)} \\
& \lesssim \left\|(-\Delta )^{s/2} ([1-\phi] f)\right\|_{L^p(\mathbb{R}^n)} \\
& \lesssim
\left\|(-\Delta )^{s/2} f\right\|_{L^p(\mathbb{R}^n)}
 +\left\|(-\Delta )^{s/2}(\phi f)\right\|_{L^p(\mathbb{R}^n)}.
\end{align*}
Combining this with the $L^p$ estimate on
$(-\Delta )^{s/2}(\phi f)$ obtained in \eqref{eq:4.1} and \eqref{eq:4.2}
completes the proof of the present lemma.
\end{proof}

\begin{prop} \label{prop:HardyForDomain}
Let $\Omega \subset \rn$ be an exterior $\cdini$ domain,  $n \ge 2$.
Assume $1<p<\infty$ and $0<s<\min\{n/p,1+1/p\}$.
Then there exists a constant $C>0$ such that, for any $f \in C_c^{\infty}(\Omega)$,
\begin{equation}\label{eq:4.3}
\left\|\frac{f}{\rho^s}\right\|_{L^p(\Omega)} \le
C\left\|(-\Delta_\Omega)^{s/2}f\right\|_{L^p(\Omega)}.
\end{equation}
\end{prop}
\begin{proof}
To prove \eqref{eq:4.3}, it suffices to show
\begin{equation}\label{eq:fir-inHardy-domain}
\left\|\frac{1}{\rho^s} (-\Delta_{\Omega})^{-\frac{s}{2}} g\right\|_{L^p(\Omega)}
\lesssim\|g\|_{L^p(\Omega)}, \quad \forall \,  g \in L^p(\Omega).
\end{equation}
Indeed, if \eqref{eq:fir-inHardy-domain} holds,
applying \eqref{eq:fir-inHardy-domain} to $g:=(-\Delta_{\Omega})^{\frac{s}{2}} f$
with $f \in C_c^{\infty}(\Omega)$, we see that \eqref{eq:4.3} holds.
We also remark that such $g:=(-\Delta_{\Omega})^{\frac{s}{2}} f$
with $f \in C_c^{\infty}(\Omega)$ do belong to $L^p(\Omega)$; see \cite[Theorem 4.3]{KVZ16}.

By Propositions \ref{riesz-potential-1} and \ref{lem:riesz-poten}, to prove
\eqref{eq:fir-inHardy-domain}, it suffices to show that the kernel
$$
K(x, y):=\frac{1}{\rho(x)^s} \frac{1}{|x-y|^{n-s}}\left(\frac{\rho(x)}{|x-y| \wedge 1} \wedge 1 \right)\left(\frac{\rho(y)}{|x-y| \wedge 1} \wedge 1\right)
$$
defines a bounded operator on $L^p(\Omega)$. To prove this conclusion,
we subdivide $\Omega \times \Omega$ into several regions,
and then applying Lemma \ref{lem:schur-test} with a suitably chosen weight,
we further obtain the $L^p$-boundedness of the operator defined by $K(x,y)$
in each of these regions. To begin, we subdivide $\Omega \times \Omega$ into two main regions: $|x-y| \leq 1$ and $|x-y|>1$.

{\it Region} $I$: $|x-y| \leq 1$. On this region, the kernel becomes
$$
K(x, y)=\frac{1}{\rho(x)^s} \frac{1}{|x-y|^{n-s}} \left(\frac{\rho(x)}{|x-y|} \wedge 1\right)\left(\frac{\rho(y)}{|x-y|} \wedge 1\right).
$$
To further analyze the kernel $K(x, y)$, we subdivide {\it Region} $I$ into four regions.

{\it Region} $Ia$: $|x-y| \leq \rho(x) \wedge \rho(y)$. In this case,
the kernel is just
$$
K(x, y)=\frac{1}{\rho(x)^s |x-y|^{n-s}}
$$
and we also have
$$
\rho(x) \leq \rho(y)+|x-y| \leq 2 \rho(y)
\quad \text { and } \quad
\rho(y) \leq \rho(x)+|x-y| \leq 2 \rho(x).
$$
An easy computation proves that
$$
\int_{|x-y| \leq \rho(x)} K(x, y) \, dy
  + \int_{|x-y| \leq \rho(y)} K(x, y) \, dx
\lesssim 1,
$$
and then the $L^p$-boundedness of the operator defined by $K(x,y)$
on this region follows immediately from Lemma \ref{lem:schur-test}.

{\it Region} $Ib$: $\rho(y) \leq |x-y| \leq \rho(x)$. On this region, the kernel has the form
$$
K(x, y) = \frac{\rho(y)}{\rho(x)^s |x-y|^{n+1-s}}.
$$
It is easy to show that
$$
\int_{\rho(y) \leq|x-y| \leq \rho(x)} K(x, y) \, dy
 \lesssim \frac{1}{\rho(x)^s} \int_{|x-y| \leq \rho(x)} \frac{1}{|x-y|^{n-s}} \, dy \lesssim 1
$$
and
$$
 \int_{\rho(y) \leq |x-y| \leq \rho(x)} K(x, y) \, dx
 \lesssim \rho(y) \int_{|x-y| \geq \rho(y)} \frac{1}{|x-y|^{n+1}} \, dx \lesssim 1 .
$$
Then, applying Lemma \ref{lem:schur-test} again, we obtain the $L^p$-boundedness
of the operator defined by $K(x,y)$ on this region.

{\it Region} $Ic$: $\rho(x) \leq |x-y| \leq \rho(y)$. On this region,
the kernel $K(x,y)$ becomes
$$
K(x, y) = \frac{\rho(x)^{1-s}}{|x-y|^{n+1-s}}
$$
and we also have
$$
|x-y| \leq \rho(y) \leq |x-y| + \rho(x) \leq 2|x-y|.
$$
To prove the $L^p$-boundedness of the operator defined by $K(x,y)$
on this region, we use Lemma \ref{lem:schur-test} with weight given by
$$
w(x, y) = \left(\frac{\rho(x)}{|x-y|}\right)^\alpha
\quad \text {with} \quad p(s-1)<\alpha<p'(2-s).
$$
The assumption $s<1+\frac{1}{p}$ guarantees the existence of such $\alpha$.
That hypothesis \eqref{eq:assump1-in-Schur} is satisfied in this case follows from
$$
\int_{\rho(x) \leq|x-y| \leq \rho(y)} w(x, y)^{\frac{1}{p}} K(x, y) \, dy
\lesssim \int_{|x-y| \geq \rho(x)} \frac{\rho(x)^{1-s+\frac{\alpha}{p}}}{|x-y|^{n+1-s+\frac{a}{p}}} \, dy \lesssim 1,
$$
while hypothesis \eqref{eq:assump2-in-Schur} is deduced from
\begin{align*}
& \int_{\rho(x) \leq|x-y| \leq \rho(y)} w(x, y)^{-\frac{1}{p'}} K(x, y) \, dx \\
& \ \lesssim \int_{\frac{1}{2} \rho(y) \leq |x-y| \leq \rho(y)}
\frac{\rho(x)^{1-s-\frac{\alpha}{p'}}} {|x-y|^{n+1-s-\frac{\alpha}{p'}}} \, dx \\
& \ \lesssim \rho(y)^{-(n+1-s-\frac{\alpha}{p'})}
\int_0^{\rho(y)} r^{1-s-\frac{\alpha}{p'}} \rho(y)^{n-1} \, dr \\
& \ \lesssim 1.
\end{align*}

{\it Region} $Id$: $\rho(x) \vee \rho(y) \leq|x-y|$. On this region,
the kernel $K(x,y)$ has the form
$$
K(x, y)=\frac{\rho(x)^{1-s} \rho(y)}{|x-y|^{n+2-s}}.
$$
To apply Lemma \ref{lem:schur-test} in this case, let
$$
w(x, y):=\left(\frac{\rho(x)}{|x-y|}\right)^\alpha
\quad \text{with} \quad p(s-1)<\alpha<p'(2-s).
$$
In this case, it is easy to find that
$$
\int_{\rho(x) \vee \rho(y) \leq |x-y|} w(x, y)^{\frac{1}{p}} K(x, y) \, dy
\lesssim
\int_{|x-y| \geq \rho(x)} \frac{\rho(x)^{1-s+\frac{\alpha}{p}}}
{|x-y|^{n+1-s+\frac{\alpha}{p}}} \, dy \lesssim 1,
$$
which implies that hypothesis \eqref{eq:assump1-in-Schur} holds in this setting.
Moreover, hypothesis \eqref{eq:assump2-in-Schur} in this setting follows from
\begin{align*}
& \int_{\rho(x) \vee \rho(y) \leq|x-y|} w(x, y)^{-\frac{1}{p'}} K(x, y) \, dx \\
& \ \lesssim \rho(y) \int_{\rho(x) \vee \rho(y) \leq |x-y|} \frac{\rho(x)^{1-s-\frac{\alpha}{p'}}} {|x-y|^{n+2-s-\frac{\alpha}{p'}}} \, dx \\
& \ \lesssim \rho(y) \sum_{R \geq \rho(y)} \frac{1}{R^{n+2-s-\frac{\alpha}{p'}}}
\int_0^{2R} r^{1-s-\frac{\alpha}{p'}} R^{n-1} \, dr \\
& \ \lesssim 1,
\end{align*}
where the sum in the second inequality is over $R \in 2^\mathbb{Z}$.

Next, we turn to the estimate on the second main region.

{\it Region} $II$: $|x-y|> 1$.
On this region, the kernel $K(x,y)$ is just
$$
K(x, y)=\frac{1}{\rho(x)^s} \frac{1}{|x-y|^{n-s}}
\left(\rho(x) \wedge 1\right)
\left(\rho(y) \wedge 1\right).
$$
Without loss of generality, we may assume that the spatial origin is the centroid
of $\Omega^c$. Furthermore, we subdivide {\it Region} $II$ into four subregions.

{\it Region} $IIa$: $1 \leq \rho(x) \wedge \rho(y)$. On this region, we have
$$
\rho(x) \sim |x| \quad \text{and} \quad \rho(y) \sim |y|,
$$
and hence
$$
K(x, y) = \frac{1}{\rho(x)^s |x-y|^{n-s}} \lesssim \frac{1}{|x|^s |x-y|^{n-s}}.
$$

To prove the $L^p$-boundedness of the operator defined by $K(x,y)$
on Region $IIa$, we apply Lemma \ref{lem:schur-test} with weight
$$
w(x, y):= \left(\frac{|x|}{|y|}\right)^\alpha \quad \text{with} \quad
ps < \alpha < p'(n-s) \wedge pn.
$$

To verify hypothesis \eqref{eq:assump1-in-Schur}, we estimate
\begin{equation} \label{eq:RegionIIa}
\int_{1 \leq \rho(x) \wedge \rho(y)} w(x, y)^{\frac{1}{p}} K(x, y) \, dy
\lesssim \int_{\mathbb{R}^n}
\frac{|x|^{\frac{\alpha}{p}-s}} {|x-y|^{n-s} |y|^{\frac{\alpha}{p}}} \, dy.
\end{equation}
Using the assumption $\frac{\alpha}{p}<n$ for $\alpha$, we find that
\begin{equation} \label{eq:RegionIIa1}
\int_{|y| \leq 2|x|}
\frac{|x|^{\frac{\alpha}{p}-s}} {|x-y|^{n-s} |y|^{\frac{\alpha}{p}}} \, dy
\lesssim \int_{|y| \leq 2|x|}
 \frac{\, dy}{|x-y|^{n-\frac{\alpha}{p}} |y|^{\frac{\alpha}{p}}}
  +\int_{|y| \leq 2|x|} \frac{\, dy}{|x-y|^{n-s}|y|^s} \lesssim 1,
\end{equation}
where to obtain the last inequality we consider separately the cases
$|y| \leq \frac{|x|}{2}$ and $\frac{|x|}{2} < |y| \leq 2|x|$;
in the former case we use $|x-y| \geq \frac{|x|}{2}$, while in the latter
we use $|x-y| \leq 3|x|$. Moreover, on the region
$\{y \in \mathbb{R}^n:\, |y| > 2|x| \}$, using the fact that $|x-y| \sim |y|$
in this case, we have
$$
\int_{|y| > 2|x|} \frac{|x|^{\frac{\alpha}{p}-s}} {|x-y|^{n-s} |y|^{\frac{\alpha}{p}}} \, dy
\lesssim |x|^{\frac{\alpha}{p}-s}
   \int_{|y|>2|x|} \frac{1}{|y|^{n-s+\frac{\alpha}{p}}} \, dy
   \lesssim 1.
$$
Combining this with \eqref{eq:RegionIIa1} and \eqref{eq:RegionIIa} yields that
hypothesis \eqref{eq:assump1-in-Schur} holds in this case.

Next, we verify hypothesis \eqref{eq:assump2-in-Schur} in this case. In this setting,
we have
\begin{equation} \label{eq:RegionIIa-2}
\int_{1 \leq \rho(x) \wedge \rho(y)} w(x, y)^{-\frac{1}{p'}} K(x, y) \, dx
\lesssim
\int_{\mathbb{R}^n} \frac{|y|^{\frac{\alpha}{p'}}} {|x|^{s+\frac{\alpha}{p'}} |x-y|^{x-s}} \, dx.
\end{equation}
Using the assumption that $s+\frac{\alpha}{p'}<n$ for $\alpha$, we conclude that
\begin{equation} \label{eq:RegionIIa-3}
\int_{|x| \leq 2|y|}
\frac{|y|^{\frac{\alpha}{p'}}} {|x|^{s+\frac{\alpha}{p'}} |x-y|^{n-s}} \, dx \lesssim
\int_{ |x| \leq 2|y|} \frac{\, dx} {|x|^s|x-y|^{n-s}} +
 \int_{|x| \leq 2|y|}\frac{\,dx}{|x|^{s+\frac{\alpha}{p'}}|x-y|^{n-s-\frac{\alpha}{p'}}} \lesssim 1,
\end{equation}
where to obtain the last inequality we consider separately the cases
$|x| \leq \frac{|y|}{2}$ and $\frac{|y|}{2}<|x| \leq 2|y|$. Moreover,
using the fact that $|x-y| \sim |x|$ on the region
$\{x \in \mathbb{R}^n:\,|x|>2|y|\}$, we obtain that
$$
\int_{|x|>2|y|}
\frac{|y|^{\frac{\alpha}{p'}}} {|x|^{s+\frac{\alpha}{p'}}|x-y|^{n-s}}\,dx
\lesssim |y|^{\frac{\alpha}{p'}}
\int_{|x|>2|y|} \frac{1}{|x|^{n+\frac{\alpha}{p'}}} \, dx
\lesssim 1,
$$
which, together with \eqref{eq:RegionIIa-2} and \eqref{eq:RegionIIa-3}, implies
that hypothesis \eqref{eq:assump2-in-Schur} holds in this case.

{\it Region} $IIb$: $\rho(y) \leq 1 \leq \rho(x)$.
On this region, $\rho(x) \sim |x|$, and hence
$$
K(x, y)=\frac{\rho(y)}{\rho(x)^s |x-y|^{n-s}} \lesssim \frac{1}{|x|^s|x-y|^{n-s}}.
$$
In this case, an argument similar to that used in the proof for Region IIa
shows the $L^p$-boundedness of the operator defined by $K(x,y)$ on Region IIb.

{\it Region} $IIc$: $\rho(x) \leq 1 \leq \rho(y)$. On this region, the kernel has the form
$$
K(x, y) = \frac{\rho(x)^{1-s}}{|x-y|^{n-s} }
$$
and we also have
$$
|x| \lesssim 1 \quad \text{and} \quad |y| \sim \rho(y).
$$
Furthermore, on Region II, $|x-y|> 1$, which implies that
$$
|y| \leq |x-y| + |x| \lesssim |x-y| + 1 \lesssim |x-y|.
$$
To obtain the $L^p$-boundedness of the operator defined by $K(x,y)$
on this region, we apply Lemma \ref{lem:schur-test} with weight defined by
$$
w(x, y):=\frac{ \rho(x)^{\alpha_2}}{|y|^\alpha},
$$
where $ps < \alpha = \alpha_1+\alpha_2 <p'(n-s)$, $\alpha_1<p$, and
$\alpha_2 < p'(2-s)$.

In this setting, it is easy to find that
\begin{align*}
 \int_{\rho(x) \leq 1 \leq \rho(y)} w(x, y)^{\frac{1}{p}} K(x, y) \,dy
& \lesssim \int_{\rho(x) \leq 1 \leq \rho(y)}
\frac{ \rho(x)^{1-s+\frac{\alpha_2}{p}}} {|y|^{\frac{\alpha}{p}}|y-x|^{n-s}} \, dy \\
& \lesssim  \rho(x)^{1-s+\frac{\alpha_2}{p}}
\int_{|y| \gtrsim \rho(x)} \frac{1}{|y|^{n-s+\frac{\alpha}{p}}} \, dy \\
& \lesssim  \rho(x)^{1-\frac{\alpha_1}{p}} \\
& \lesssim 1,
\end{align*}
which implies that hypothesis \eqref{eq:assump1-in-Schur} holds in this case.
Moreover, we also have
\begin{align*}
 \int_{\rho(x) \leq 1 \leq \rho(y)} w(x, y)^{-\frac{1}{p'}} K(x, y) \, dx
& \lesssim \int_{\rho(x) \leq 1} \frac{\rho(x)^{1-s-\frac{\alpha_2}{p'}}|y|^{\frac{\alpha}{p'}}}
     { |x-y|^{n-s}} \, dx \\
& \lesssim \int_{\rho(x) \leq 1} \frac{\rho(x)^{1-s-\frac{\alpha_2}{p'}}}{|x-y|^{n-s-\frac{\alpha}{p'}}} \, dx \\
& \lesssim \int_{\rho(x) \leq 1} \rho(x)^{1-s-\frac{\alpha_2}{p'}} \, dx \\
& \lesssim
\int_0^{1} r^{1-s-\frac{\alpha_2}{p'}} \, dr \\
& \lesssim 1,
\end{align*}
which implies that hypothesis \eqref{eq:assump2-in-Schur} holds in this setting.

{\it Region} $IId$: $\rho(x) \vee \rho(y) \leq 1$. On this region,
we have $|x| \lesssim 1$, $|y| \lesssim 1$, and
$$
1 < |x-y| \lesssim \rho(x)+\rho(y) + 1  \lesssim 1.
$$
Meanwhile, the kernel $K(x,y)$ has the form
$$
K(x, y) = \frac{\rho(x)^{1-s} \rho(y)} {|x-y|^{n-s}}
      \lesssim \rho(x)^{1-s}.
$$
To obtain the $L^p$-boundedness of the operator defined by $K(x,y)$
on Region IId, we apply Lemma \ref{lem:schur-test} with weight given by
$$
w(x, y):=\left(\frac{\rho(x)}{|y|}\right)^\alpha
\quad \text{with} \quad p(s-1) < \alpha < p'(2-s).
$$
In this setting, it is easy to see that
$$
\int_{\rho(x) \vee \rho(y) \leq 1} w(x, y)^{\frac{1}{p}} K(x, y) \, dy
 \lesssim \int_{|y| \lesssim 1} \frac{\rho(x)^{1-s+\frac{\alpha}{p}}}
 {|y|^{\frac{\alpha}{p}} } \, dy \lesssim 1,
$$
which yields that hypothesis \eqref{eq:assump1-in-Schur} holds in this case.
Moreover, we have
\begin{align*}
\int_{\rho(x) \vee \rho(y) \leq 1} w(x, y)^{-\frac{1}{p}} K(x, y) \, dx
& \lesssim \int_{\rho(x) \leq 1} |y|^{\frac{\alpha}{p'}} \rho(x)^{1-s-\frac{\alpha}{p'}} \, dx \\
& \lesssim\int_{\rho(x) \leq 1} \rho(x)^{1-s-\frac{\alpha}{p'}} \, dx \\
& \lesssim \int_0^{1} r^{1-s-\frac{\alpha}{p'}} \, dr \\
& \lesssim 1,
\end{align*}
which further implies that hypothesis \eqref{eq:assump2-in-Schur} holds in this case.

Putting everything above together, we conclude that
the operator defined by $K(x,y)$ is bounded on $L^p(\Omega)$.
This completes the proof of Proposition \ref{prop:HardyForDomain}.
\end{proof}

\subsection{The difference of heat kernels on $\rn$ and the domain}\hskip\parindent
We have the following upper bound of the difference $p(t,x,y)-p_\Omega(t,x,y)$.
\begin{thm}\label{heat-comparison}
Let $\Omega\subset \rn$ be an open and connected set with boundary $\partial \Omega$, $n\ge 2$. Then there exist $c,C>1$ such that,
for all $x,y\in \rn$ and all $t>0$,
it holds that
\begin{equation}\label{upper-difference}
0\le p(t,x,y)-p_\Omega(t,x,y)\le  Ct^{-n/2}e^{-\frac{|x-y|^2+\rho(x)^2+\rho(y)^2}{ct}}.
\end{equation}
\end{thm}

\begin{proof}
For the case $x$ or $y$ belongs to $\rn\setminus \Omega$,
it holds $p_\Omega(t,x,y)=0$.
Assume $x\in \rn\setminus \Omega$, which implies that $\rho(x)=0$ and
$$\rho(y)\le |x-y|.$$
Therefore, we have
\begin{equation*}\label{heat-diff-1}
p(t,x,y)-p_\Omega(t,x,y)=p(t,x,y)\le  Ct^{-n/2}e^{-\frac{|x-y|^2+\rho(x)^2+\rho(y)^2}{ct}}.
\end{equation*}

It remains to consider the case $x,y\in \Omega$.
Without loss of generality, we may suppose that $\rho(y)\le \rho(x)$.
If $|x-y|\ge \frac{1}{2 \sqrt n} \rho(x),$
then it holds that
$\rho(y)\le \rho(x)\le 2\sqrt n|x-y|,$
which implies that
\begin{equation*}\label{heat-diff-2}
0\le p(t,x,y)-p_\Omega(t,x,y)\le   Ct^{-n/2}e^{-\frac{|x-y|^2}{ct}}\le  Ct^{-n/2}e^{-\frac{|x-y|^2+\rho(x)^2+\rho(y)^2}{ct}}.
\end{equation*}

\begin{figure}[htbp]
\centering
\includegraphics[scale=0.25]{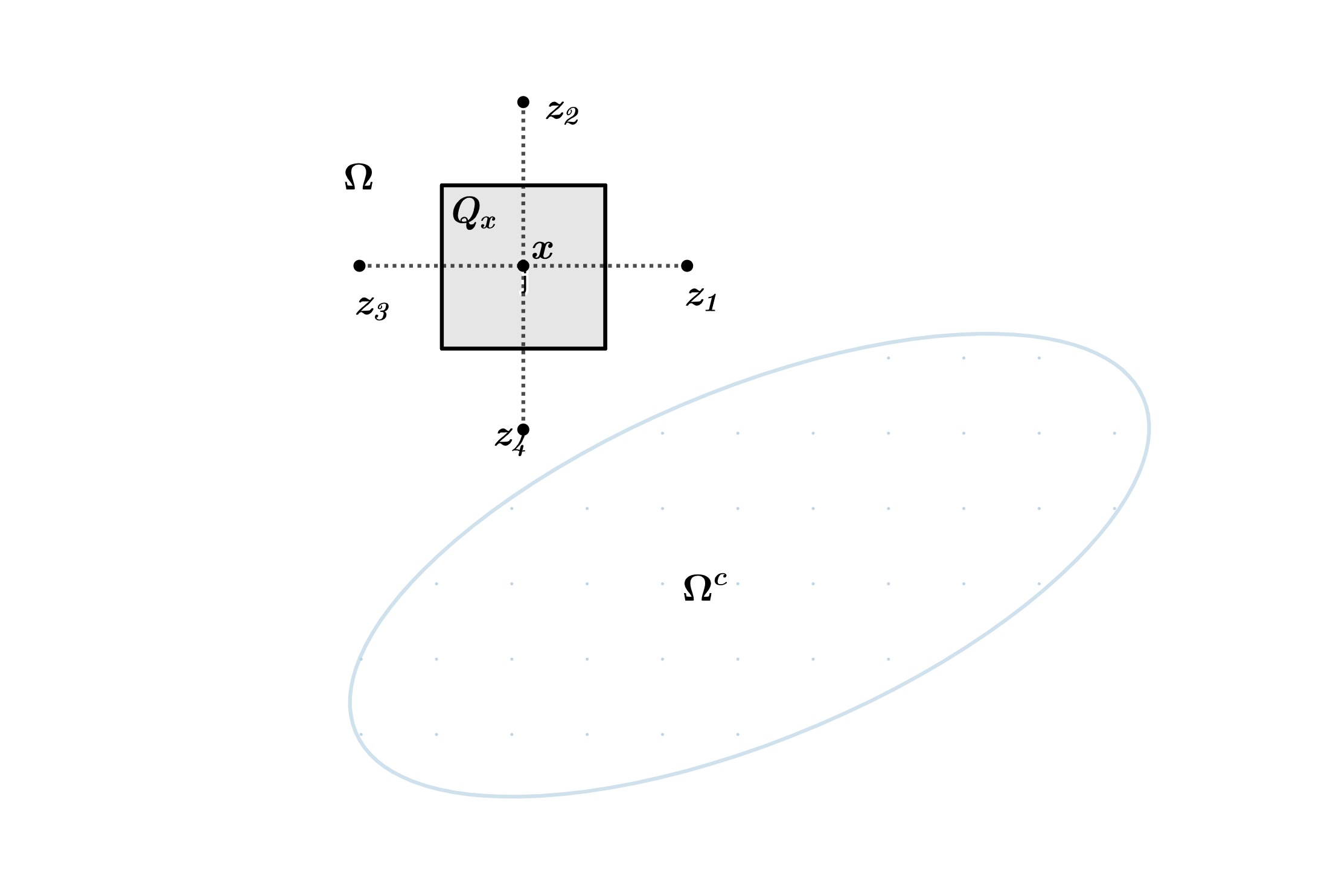}
\caption{Example of the cube in the two-dimensional case}
\label{fig:GjIj}
\end{figure}

If $|x-y|< \frac{1}{2 \sqrt n} \rho(x),$ let $Q_x$ be the cube with centre at
$x \in \Omega$ and side length $\frac{1}{\sqrt n} \rho(x)$.
Then it holds that $B(x, \frac{\rho(x)}{2\sqrt{n}} ) \subset Q_x$; see Picture \ref{fig:GjIj}.
The cube $Q_x$ has $2 n$ faces, and let $I$ be the collection of the $2n$ points that
are $\frac{1}{\sqrt n} \rho(x)$ away from $x$ though the centre of
those faces, which means that every $z_i \in I$ is the reflection of $x$
by one of the faces. Then the function
$$u(t,x,y) = p(t,x,y) - \sum_{z_i \in I} p(t,z_i,y)$$
satisfies the heat equation in $Q_x$
  with the same initial condition $\delta(x-y)$,
  since
$$\lim_{t \rightarrow 0} [p(t,x,y) - \sum_{z_i \in I} p(t,x,z_i)] = \delta(x-y) - \sum_{z_i \in I} \delta(y-z_i) = \delta(x-y).$$

Moreover, for all $t >0$, $u(t,x,y)$ is negative on the boundary of $Q_x$, since
$p(t,x,y) - p(t,z_i,y)$ is the heat kernel of the $p_{\mathbb{H}_{z_i}}$, where $\mathbb{H}_{z_i}$ is the half-space
containing $x$ and its boundary containing the face between $x$ and $z_i$. Hence, $p(t,x,y) - p(t,z_i,y) = 0$ in this face.

Via the maximal principle, for any $y \in Q_x$, it holds that
\[
p_{Q_x}(t,x,y) \ge u(t,x,y) = p(t,x,y) - \sum_{z_i \in I} p(t,z_i,y).
\]
If $y \in Q_x$, then it holds that
\[
p(t,x,y)-p_{\Omega}(t,x,y)\leq p(t,x,y) - p_{Q_x}(t,x,y)
\leq \sum_{z_i \in I} p(t,z_i,y).
\]
Moreover, as $\frac{1}{2\sqrt n}\rho(x) \leq |x- z_i|$, we see that
\[
p(t,x,y)-p_{\Omega}(t,x,y)\leq Ct^{-n/2} e^{-\frac{\rho(x)^2}{ct}},
\]
which together with  $|x-y|< \frac{1}{2 \sqrt n} \rho(x)$ and $\rho(y) \leq \rho(x)$ yields that
\[
p(t,x,y)-p_{\Omega}(t,x,y)\leq C t^{-n/2} e^{-\frac{|x-y|^2 + \rho(x)^2 + \rho(y)^2}{ct}}.
\]
This completes the proof.
\end{proof}

\subsection{Main estimates}

\begin{thm}\label{littlewood-paley}
Suppose $\Omega \subset \mathbb{R}^n$ is a domain, $n \ge 2$. Assume $1< p < \infty$, $s \ge 0$, and $k \ge 1$ is an integer satisfying $2k >s$.
For any $N\in(0,\infty)$, let $\tilde{P}_N := e^{\Delta/N^2} - e^{4 \Delta/N^2}$
and $\tilde{P}^\Omega_N := e^{\Delta_\Omega/N^2} - e^{4 \Delta_\Omega/N^2}$.
The following statements are valid.
	
(i) For any $f\in C^\infty_c(\rn)$, it holds that
\begin{eqnarray*}
\left\|(-\Delta)^{s/2}f\right\|_{L^p(\rn)}\sim \left\|\left(\sum_{N \in 2^{\mathbb{Z}}} N^{2s}
\left|\left(\tilde{P}_N \right)^k f\right|^2\right)^{1/2}\right\|_{L^p(\rn)}.
\end{eqnarray*}

(ii) For any $f \in C^\infty_c(\Omega)$, it holds that
\begin{eqnarray*}
\left\|(-\Delta_\Omega)^{s/2} f\right\|_{L^p(\Omega)}\sim \left\|\left(\sum_{N \in 2^{\mathbb{Z}}} N^{2s}
\left| \left(\tilde{P}^\Omega_N \right)^k f\right|^2\right)^{1/2}\right\|_{L^p(\Omega)}.
\end{eqnarray*}
\end{thm}
\begin{proof}
For the proof, we refer to \cite[Theorem 4.3]{KVZ16}.
\end{proof}

\begin{lem} \label{lem:LP-kernel-upperbound}
Let $\Omega \subset \mathbb{R}^n$ be an exterior Lipschitz domain,
$n \ge 2$, and $N\in(0,\infty)$. For an integer $k \geq 1$, let $K_N^k(x, y):=[(\tilde{P}_N)^k - (\tilde{P}_N^{\Omega})^k](x, y)$.
Then there exists a constant $c=c(k)>0$ such that
\begin{align} \label{eq:kernel-bound}
|K_N^k(x, y)|\lesssim_k N^n e^{-c N^2[\rho(x)^2 + \rho(y)^2 + |x-y|^2]}
\end{align}
uniformly for any $x \in \mathbb{R}^n$ and $y \in \Omega$, where the implicit positive constant depends on $k$.
\end{lem}
\begin{proof}
 We first write
\begin{align} \label{eq:factor}
K_N^k(x, y)
& =\left[e^{\Delta/N^2} - e^{4\Delta/N^2}\right]^k(x, y)-
   \left[e^{\Delta_{\Omega}/N^2} - e^{4\Delta_{\Omega}/N^2}\right]^k(x, y)\\   \notag
& =\left\{ \sum_{\ell=0}^{k-1} \left[e^{\Delta/N^2} - e^{4\Delta/N^2}\right]^{\ell}
 \left[e^{\Delta/N^2} - e^{\Delta_{\Omega}/N^2} - e^{4\Delta/N^2}
      + e^{4\Delta_{\Omega}/N^2}\right] \right.\\ \notag
& \quad\times \left[e^{\Delta_{\Omega}/N^2} -
        e^{4\Delta_{\Omega}/N^2}\right]^{k-\ell-1}\Bigg\}(x, y).
\end{align}
To proceed, we will estimate the kernels of each of the three factors appearing
in the right-hand of \eqref{eq:factor}.

Using the maximal principle, a simple application of Gaussian integrals yields that
\begin{equation} \label{eq:fir-thrid}
\sup_{0 \leq \ell \leq k-1}
\left|\left[e^{\Delta_{\Omega}/N^2} - e^{4\Delta_\Omega/N^2}\right]^{\ell}(x, y)\right|
+\left|\left[e^{\Delta/N^2} - e^{4\Delta/N^2}\right]^{\ell}(x, y)\right|\\ \notag
\lesssim_k N^n e^{-c_1 N^2|x-y|^2},
\end{equation}
for some $c_1=c_1(k)>0$.

We turn now to estimating the kernel of $e^{t\Delta} - e^{t\Delta_\Omega}$.
When $y \notin \Omega$, we have
\begin{equation} \label{eq:comp-ninOmega}
0 \leq [e^{t\Delta} - e^{t\Delta_\Omega}](x, y)
= e^{t \Delta}(x, y) \lesssim t^{-n/2} e^{-\frac{|x-y|^2}{4t}}.
\end{equation}
Moreover, when $y \in \Omega$, applying Theorem \ref{heat-comparison}, we have
\begin{equation} \label{eq:comp-inOmega}
0\le  [e^{t\Delta} - e^{t\Delta_\Omega}](x, y) \le
Ct^{-n/2}e^{-\frac{|x-y|^2+\rho(x)^2+\rho(y)^2}{ct}}.
\end{equation}
Combining \eqref{eq:factor}-\eqref{eq:comp-inOmega}, for $x \in \mathbb{R}^n$
and $y \in \Omega$, we conclude that
\begin{align} \label{eq:lp-com-kern}
|K_N^k(x, y)|
& \lesssim  \iint_{\mathbb{R}^n \times \Omega^c} N^{3n}
e^{-c_2 N^2 |x-x'|^2 - c_2 N^2 |x' - y'|^2 - c_2 N^2|y-y'|^2} \, dx' \, dy' \\ \notag
&\quad + \iint_{\mathbb{R}^n \times \Omega} N^{3n}
e^{-c_2 N^2 |x-x'|^2 - c_2 N^2[\rho(x')^2 + \rho(y')^2 + |x'-y'|^2]
- c_2 N^2  |y'-y|^2} \, dx' \, dy',
\end{align}
for some $0 < c_2 \leq \min\{c_1, \frac{1}{100}\}$.

We now estimate the first integral appearing in \eqref{eq:lp-com-kern}.
As $y \in \Omega$, it holds that $|y-y'| \geq \rho(y)$.
Also,
$$
\rho(x) \leq |x-x'| + |x'-y'| + |y'-y| + \rho(y).
$$
Thus, we can bound the first integral in \eqref{eq:lp-com-kern} by
\begin{align} \label{eq:lp-com-kern-1}
& N^{3n} e^{-c_3 N^2 [\rho(x)^2 + \rho(y)^2]} \iint_{\mathbb{R}^n \times \Omega^c}
e^{-c_3 N^2(|x-x'|^2 + |x'-y'|^2 + |y'-y|^2)} \, dx' \, dy' \\ \notag
&\quad \lesssim N^n e^{-c N^2 [\rho(x)^2 + \rho(y)^2 + |x-y|^2]}.
\end{align}
To estimate the second integral in \eqref{eq:lp-com-kern}, we argue similarly and use $\rho(x) \leq \rho(x') + |x-x'|$ and $\rho(y) \leq \rho(y') + |y-y'|$ to obtain the bound
$$
\iint_{\mathbb{R}^n \times \Omega}
 N^{3n} e^{-c_2 N^2 |x-x'|^2 - c_2 N^2 [\rho(x)^2 + \rho(y')^2 + |x'-y'|^2]
 - c_2 N^2|y-y'|^2} \, dx' \, dy'
 \lesssim N^n e^{-c N^2 [\rho(x)^2 + \rho(y)^2 + |x-y|^2]}.
$$
This, combined with \eqref{eq:lp-com-kern} and \eqref{eq:lp-com-kern-1},
further implies that the estimate \eqref{eq:kernel-bound} holds.
This completes the proof of the lemma.
\end{proof}

\begin{prop}\label{prop:compare}
Let $\Omega \subset  \mathbb{R}^n$ be an exterior Lipschitz domain, $n \ge 2$.
Assume $1<p<\infty$ and $s>0$.
Then for any $f \in C_c^{\infty}(\Omega)$, we have
\begin{align}\label{eq:lp-hardy}
\left\|\left(\sum_{N \in 2^\mathbb{Z}} N^{2s}
   \left|(\tilde{P}_N)^k f\right|^2\right)^{\frac{1}{2}}
- \left(\sum_{N \in 2^\mathbb{Z}} N^{2s}
     \left|(\tilde{P}_N^{\Omega})^k f\right|^2\right)^{\frac{1}{2}}\right\|
     _{L^p(\mathbb{R}^n)}
\lesssim_k \left\|\frac{f}{\rho^s}\right\|_{L^p(\Omega)}
\end{align}
for any integer $k \geq 1$, where the implicit positive constant depends on $k$.
\end{prop}
\begin{proof}
By the triangle inequality, it holds
\begin{align*}
\text{LHS of \eqref{eq:lp-hardy}}
& \lesssim \left\| \left( \sum_{N \in 2^\mathbb{Z}} N^{2 s}
\left|\left[(\tilde{P}_N)^k-(\tilde{P}_N^{\Omega})^k \right]f\right|^2 \right)^{\frac{1}{2}} \right\|_{L^p(\mathbb{R}^n)} \\
& \lesssim \left\|\left(\sum_{N \in 2^\mathbb{Z}} N^{2s}
      \left| \int_{\Omega} K_N^k(x, y) f(y) \, dy \right|^2 \right)^{\frac{1}{2}}\right\|_{L^p(\mathbb{R}^n)} \\
& \lesssim \left\|\sum_{N \in 2^\mathbb{Z}} N^s \int_{\Omega}|K_N^k(x, y)| |f(y)| \,dy \right\|_{L^p(\mathbb{R}^n)} \\
& \lesssim \left\|\int_{\Omega}\left(\sum_{N \in 2^\mathbb{Z}} N^s |K_N^k(x, y)|\right) |f(y)| \, dy \right\|_{L^p(\mathbb{R}^n)}.
\end{align*}
Using Lemma \ref{lem:LP-kernel-upperbound}, we obtain that, for any $x\in\mathbb{R}^n$ and $y\in\Omega$,
\begin{align*}
\sum_{N \in 2^\mathbb{Z}} N^s |K_N^k(x, y)|
& \lesssim_k \sum_{N \in 2^\mathbb{Z}} N^{n+s}
                 e^{-c N^2 [\rho(x)^2 + \rho(y)^2 + |x-y|^2]} \\
& \lesssim_k \sum_{N^2 \leq [\rho(x)^2 + \rho(y)^2 + |x-y|^2]^{-1}} N^{n+s} \\
&\quad + \sum_{N^2 > [\rho(x)^2 + \rho(y)^2 + |x-y|^2]^{-1}} \frac{N^{n+s}}{\left(N^2[\rho(x)^2+\rho(y)^2+|x-y|^2]\right)^{n+s}} \\
& \lesssim_k [\rho(x)^2+\rho(y)^2+|x-y|^2]^{-\frac{n+s}{2}}.
\end{align*}
Therefore, to prove the proposition, it suffices to show that the kernel
$K:\,\mathbb{R}^n \times \Omega \rightarrow \mathbb{R}$ given by
$$
K(x, y):=\rho(y)^s[\rho(x)^2+\rho(y)^2+|x-y|^2]^{-\frac{n+s}{2}}
$$
defines an operator bounded from $L^p(\Omega)$ to $L^p(\mathbb{R}^n)$.
To establish this, we will apply Lemma \ref{lem:schur-test} with weight given by
$$
w(x, y):=\left(\frac{\rho(x)}{\rho(y)}\right)^\alpha
\quad \text{with} \quad 0 < \alpha < p^{\prime} \wedge ps.
$$

We first verify that hypothesis \eqref{eq:assump1-in-Schur} holds in this setting; indeed,
\begin{align}\label{eq:wk-2}
\int_{\Omega} w(x, y)^{\frac{1}{p}} K(x, y) \, dy
& = \int_{\Omega} \frac{\rho(x)^{\frac{\alpha}{p}} \rho(y)^{s-\frac{\alpha}{p}}}
           {[\rho(x)^2 + \rho(y)^2 + |x-y|^2]^{\frac{n+s}{2}}}\,dy \\ \notag
& \lesssim \int_{\Omega} \frac{\rho(x)^{\frac{\alpha}{p}}}
{[\rho(x) + |x-y|]^{n+\frac{\alpha}{p}}} \, dy \lesssim 1,
\end{align}
where in order to obtain the last inequality, we consider separately the regions
$|y-x| \leq \rho(x)$ and $|y-x| > \rho(x)$.

Next, we verify that hypothesis \eqref{eq:assump2-in-Schur} holds in this setting. Notice that
\begin{equation} \label{eq:wk}
\int_{\mathbb{R}^n} w(x, y)^{-\frac{\alpha}{p'}} K(x, y)\,dx
=\int_{\mathbb{R}^n}
\frac{\rho(y)^{s + \frac{\alpha}{p'}}}
{\rho(x)^{\frac{\alpha}{p'}} [\rho(x)^2 + \rho(y)^2 + |x-y|^2]^{\frac{n+s}{2}}} \, dx.
\end{equation}
On the region that $|x-y| \leq \rho(y) / 2$, we have $\rho(x) \sim \rho(y)$. Thus, we may bound the contribution of this region
to the right-hand side of \eqref{eq:wk} by
$$
\rho(y)^{-n} \int_{|x-y| \leq \frac{1}{2} \rho(y)} \, dx \lesssim 1.
$$
Moreover, the contribution of the region that $|x-y|>\rho(y) / 2$ and $\rho(x)>\rho(y)$ to the right-hand side of \eqref{eq:wk}
is bounded by
$$
\rho(y)^s \int_{|x-y|>\rho(y) / 2} \frac{1}{|x-y|^{n+s}} \, dx \lesssim 1.
$$
Finally, we estimate the contribution of the region that $|x-y| > \rho(y) / 2$ and $\rho(x) \leq \rho(y)$ to the right-hand side
of \eqref{eq:wk} by
$$
\rho(y)^{s+\frac{\alpha}{p'}} \sum_{R \geq \rho(y)} R^{-n-s}
\int_{|x-y| \sim R} \frac{\, dx}{\rho(x)^{\frac{\alpha}{p'}}}
\lesssim \rho(y)^{s + \frac{\alpha}{p'}} \sum_{R \geq \rho(y)} R^{-n-s} R^{n-\frac{\alpha}{p'}} \lesssim 1 .
$$
Therefore, we find that
\begin{equation} \label{eq:wk-1}
\int_{\mathbb{R}^n} w(x, y)^{-\frac{\alpha}{p'}} K(x, y)\,dx\lesssim1.
\end{equation}
By \eqref{eq:wk-2} and \eqref{eq:wk-1}, and applying Lemma \ref{lem:schur-test},
we further conclude that \eqref{eq:lp-hardy} holds.
This completes the proof of the proposition.
\end{proof}

\begin{thm}
Suppose that $\Omega \subset \mathbb{R}^n$
is an exterior $\cdini$ domain, $n \ge 2$.
Then, for any $1<p<\infty$ and $0<s<\min\{n/p, 1+1/p\}$, it holds that
for all $f\in C^\infty_c(\Omega)$,
$$\left\|(-\Delta )^{s/2}f\right\|_{L^p(\rn)}\sim\left\|(-\Delta_\Omega)^{s/2}f\right\|_{L^p(\Omega)}.$$
\end{thm}

\begin{proof}
Fix $f \in C_c^{\infty}(\Omega)$ and choose an integer $k \geq 1$ such that $2k > s$. Using Theorem \ref{littlewood-paley},
the triangle inequality, Proposition \ref{prop:compare}, and Proposition \ref{prop:HardyForDomain}, we conclude that
\begin{align*}
\left\|(-\Delta )^{s/2} f\right\|_{L^p(\mathbb{R}^n)}
& \sim \left\| \left(\sum_{N \in 2^\mathbb{Z}} N^{2s}
          \left|(\tilde{P}_N)^k f\right|^2\right)^{\frac{1}{2}}\right\|_{L^p(\mathbb{R}^n)} \\
& \lesssim \left\|\left(\sum_{N \in 2^\mathbb{Z}}
N^{2s} \left|(\tilde{P}_N^{\Omega})^k f\right|^2\right)^{\frac{1}{2}}\right\|_{L^p(\Omega)}
  + \left\|\left(\sum_{N \in 2^\mathbb{Z}}
   N^{2s}\left|\left[(\tilde{P}_N)^k - (\tilde{P}_N^{\Omega})^k\right]f\right|^2\right)^{\frac{1}{2}}\right\|_{L^p(\mathbb{R}^n)} \\
& \lesssim \left\|(-\Delta_{\Omega})^{s/2} f\right\|_{L^p(\Omega)}
               + \left\|\frac{f}{\rho^s}\right\|_{L^p(\Omega)} \\
& \lesssim \left\|(-\Delta_{\Omega})^{s/2} f\right\|_{L^p(\Omega)}.
\end{align*}

Arguing similarly and using Proposition \ref{prop:HardyForWhole} in place of Proposition \ref{prop:HardyForDomain}, we also obtain
\begin{align*}
\left\|(-\Delta_{\Omega})^{s/2} f\right\|_{L^p(\Omega)}
& \sim \left\|\left(\sum_{N \in 2^\mathbb{Z}}
N^{2s} \left|(\tilde{P}_N^{\Omega})^k f\right|^2\right)^{\frac{1}{2}}\right\|_{L^p(\Omega)} \\
& \lesssim \left\|\left(\sum_{N \in 2^\mathbb{Z}}
  N^{2s} \left|(\tilde{P}_N)^k f\right|^2\right)^{\frac{1}{2}}\right\|_{L^p(\Omega)}
+ \left\|\left(\sum_{N \in 2^\mathbb{Z}}
    N^{2s} \left|\left[(\tilde{P}_N)^k - (\tilde{P}_N^{\Omega})^k\right] f\right|^2\right)^{\frac{1}{2}}\right\|_{L^p(\Omega)} \\
& \lesssim\left\|(-\Delta )^{s/2} f\right\|_{L^p(\mathbb{R}^n)}
   + \left\|\frac{f}{\rho^s}\right\|_{L^p(\Omega)} \\
& \lesssim \left\|(-\Delta )^{s/2} f\right\|_{L^p(\mathbb{R}^n)}.
\end{align*}
This completes the proof of the theorem.
\end{proof}

\begin{rem}\label{unboundedness}\rm
The same argument as in \cite[Proposition 7.1]{KVZ16} together with the heat kernel estimate obtained in Theorem \ref{main-heat-kernel}
shows that Theorem \ref{main-1} does not hold for $s\ge 1+1/p$, where $1<p<\infty$.

For the case that $2/p<s<1+1/p$ with $1<p<\infty$,
following \cite[Proposition 7.2]{KVZ16}, let us set $\Omega:=\{x\in\rr^2:\,|x|>1\}$ and set for $R>1$ that
$$f_R(x):=\phi_R(x)\log |x|,$$
where $\phi_R$ is a smooth function satisfying that
$$\phi_R(x)=\begin{cases}\log(R/|x|)(\log R)^{-1},\, & 1\le |x|\le R/2,\\
0,\, & |x|\ge R,
\end{cases}$$
and
$$|\partial_x^\alpha\phi_R(x)|\lesssim_\alpha R^{-|\alpha|}(\log R)^{-1}\,  \ \mbox{for}\, \  R/2\le |x|\le R,$$
for all multi-indices $\alpha$ with $|\alpha|\ge 0$.

Then it holds for $R>2$ and $p>1$ that
$$\|f_R\|_{L^p(\Omega)}\lesssim R^{\frac{2}{p}}\log R$$
and
$$\left\|\Delta_\Omega f_R\right\|_{L^p(\Omega)}\lesssim R^{\frac{2}{p}-2}.$$
By the boundedness of $(-\Delta_\Omega)^{it}$ on $L^p(\Omega)$ from \cite[Theorem 2]{siw01}
and the complex interpolation theorem, we conclude for $0<s<2$ that
$$\left\|(-\Delta_\Omega)^{s/2} f_R\right\|_{L^p(\Omega)}\lesssim R^{\frac{2}{p}-s}(\log R)^{\frac{2-s}{2}}.$$
This further implies for $p>2/s$ that
\begin{align}\label{non-equivalence-1}
\left\|(-\Delta_\Omega)^{s/2} f_R\right\|_{L^p(\Omega)}\to 0
\end{align}
as $R\to \infty$.

On the other hand, the same proof as in \cite[Proposition 7.2]{KVZ16} proves that
\begin{align}\label{non-equivalence-2}
\left\|(-\Delta )^{s/2} f_R\right\|_{L^p(\rn)}\gtrsim 1
\end{align}
uniformly for $R>2$.

Therefore, the two estimates \eqref{non-equivalence-1} and \eqref{non-equivalence-2} show that
Theorem \ref{main-1} does not hold for $2/p<s<1+1/p$ on exterior domains in $\rr^2$.
\end{rem}

\section{An application to NLS}\label{s5} \hskip\parindent
In this section, we apply our main result to the NLS
\begin{align}\label{nls-2d}
\begin{cases}
  i\partial_t u=-\Delta_{\Omega} u \pm |u|^pu\ \text{in}\ \Omega\times I,\\
u(x,0)=u_0(x)\ \text{in}\ \Omega,\\
u(x,t)|_{\partial\Omega}=0.
\end{cases}
\end{align}


Here $p\in(1,\infty)$ and $I$ denotes a time interval containing the origin.
For $n\ge 3$, the energy critical case was obtained in \cite{KVZ16} with $p=\frac{4}{n-2}$.

Our result will mainly be applicable in the two dimensional exterior domains for the $\dot{H}^s$-critical case, where
\begin{align*}
\begin{cases}
 s=\frac{n}{2}-\frac{2}{p}, \\
0<s\le 1,\\
0<p<\infty.
\end{cases}
\end{align*}

We shall need the following Strichartz estimate from \cite{I10}.

\begin{thm}\label{Strichartz}
Let $n\ge 2$, $\Omega\subset\mathbb{R}^n$ be the exterior of a smooth compact strictly convex obstacle, and $I$ be a finite
time interval containing the origin.
Let $q,\tilde{q}>2$, $2\le r,\tilde r\le \infty$ satisfy the scaling conditions
$$\frac{2}{q}+\frac{n}{r}=\frac{n}{2}=\frac{2}{\tilde q}+\frac{n}{\tilde r}, \, (n,q,r)\neq (2,2,\infty)\neq (n,\tilde q,\tilde r).$$
Then it holds that
\begin{align*}
\left\|e^{it\Delta_{\Omega}}u_0-i\int_0^t e^{i(t-s)\Delta_{\Omega}}F(s)\,ds\right\|_{L^q_tL^r_x(I\times \Omega)}
\lesssim \|u_0\|_{L^2(\Omega)}+\|F\|_{L^{\tilde{q}'}_tL^{\tilde{r}'}_x(I\times\Omega)}
\end{align*}
with the implicit positive constant independent of the time interval $I$.
\end{thm}
Note that in the above theorem
$$
u(t) = e^{it\Delta_{\Omega}}u_0-i\int_0^t e^{i(t-s)\Delta_{\Omega}}F(s)
$$
represents the proper formulation of the solution to
$$
    i \partial_t u(x,t) = -\Delta u(t,x) + F(t,x)
   \quad \text{with }
u(x,0)=u_0(x)
\quad \text{and}\quad
u(x,t)|_{\partial\Omega}=0.
$$

In what follows, we fix an  $r$ as
$$r:=\frac{2n^2+2(4-2s)n}{n^2-2sn+8s}$$
and $q$ such that
$$\frac{2}{q}+\frac{n}{r}=\frac{n}{2},$$
where $0<s<1$ for $n=2$.
Moreover we choose the indexes $\tilde{q},\tilde{r}$ satisfying
\begin{align}\label{tilde-q}
q=(1+p)\tilde{q}'
\end{align}
and
\begin{align}\label{tilde-r}
\frac{r}{\tilde{r}'}=\frac{4(n-sr)}{n(n-2s)}+1,
\end{align}
i.e.,
$$\tilde{r}'=\frac{r}{\frac{4(n-sr)}{n(n-2s)}+1}=\frac{rn(n-2s)}{n^2+(4-2s)n-4sr}.$$

Note that as  $0<s<1$ for $n=2$, one has
$$2<r<\frac{n}{s}.$$
Moreover, it holds that
$$\frac 2q=\frac n2-\frac nr=\frac n2- \frac{n^2-2sn+8s}{2n+2(4-2s)}
=\frac{n^2+(4-2s)n-(n^2-2sn+8s)}{2n+2(4-2s)}=\frac{2n-4s}{n+(4-2s)}<1,$$
and therefore $q>2$.

Regarding the parameters $\tilde{q},\tilde{r}$, note that
\begin{align*}
\frac{2}{\tilde{q}}+\frac{n}{\tilde{r}}=2(1-\frac{1}{\tilde{q}'})+n(1-\frac{1}{\tilde{r}'})
&=2(1-\frac{1+p}{q})+n(1-\frac 1r-\frac{4(n-sr)}{nr(n-2s)})\nonumber\\
&=2+n-\frac{2}{q}-\frac{n}{r} -\frac{2p}{q}-\frac{4(n-sr)}{r(n-2s)}\nonumber\\
&=\frac n2+\frac{2rn-4n}{r(n-2s)}-\frac{4}{(n-2s)}\frac{n(r-2)}{2r}\nonumber\\
&=\frac  n2,
\end{align*}
$$\tilde{q}'=\frac{q}{1+p}=\frac{n+(4-2s)}{n-2s} \frac{n-2s}{n-2s+4}=1,$$
and
$$\tilde{r}'=\frac{r}{\frac{4(n-sr)}{n(n-2s)}+1}=\frac{rn(n-2s)}{n^2+(4-2s)n-4sr}=2.$$

Following \cite{KV13,KVZ16}, we have the following strong form of local well-posedness of
the equation \eqref{nls-2d}.

\begin{thm}\label{NLS-2D}
Let $s\in (0,1)$, $p:=4/(2-2s)$, $r:=({6-2s})/({1+s})$, and $q:=({3-s})/({1-s})$. Let $\Omega\subset\rr^2$ be the exterior of a
smooth compact strictly convex obstacle. There exists $\eta>0$ such that if $u_0\in H^s_D(\Omega)$ satisfies
$$\left\|(-\Delta_{\Omega})^{\frac s2}e^{it\Delta_{\Omega}}u_0\right\|_{L_t^{q}L_x^r(I\times\Omega)}\le \eta$$
for some time interval $I$ containing $0$, then there is a unique strong $C_t^0\dot{H}_D^s(I\times\Omega)$  solution to
the equation \eqref{nls-2d}, and it holds that
$$\left\|(-\Delta_{\Omega})^{\frac s2}u\right\|_{L_t^{q}L_x^r(I\times\Omega)}\lesssim \eta$$
with the implicit positive constant independent of the time interval $I$.
\end{thm}
\begin{proof}
Let $\tilde q$ and $\tilde r$ be defined as in \eqref{tilde-q} and \eqref{tilde-r}, respectively.
Note that $\tilde r=2$ and $\tilde q=\infty.$

Consider the mapping given by
\begin{equation*}
\Phi(u):=e^{it\Delta_{\Omega}}u_0 \mp i\int_0^t e^{i(t-s)\Delta_{\Omega}} |u(s)|^pu(s)\,ds.
\end{equation*}
Let
\begin{align*}
B_1:=\left\{u\in L^\infty_t{H}^s(I\times\Omega):\,\|u\|_{L^\infty_t{H}^s_x(I\times\Omega)}\le 2\|u_0\|_{H_x^s(\Omega)}+C(\Omega)(2\eta)^{1+p}\right\}
\end{align*}
and
\begin{align*}
B_2:=\left\{u\in L^q_tH^{s,r}(I\times\Omega):\,
\|u\|_{L^q_tL_x^{r}(I\times\Omega)}\le 2C(\Omega)\|u_0\|_{L^2_x(\Omega)}
\, \text{and} \,
\|(-\Delta_{\Omega})^{\frac s2}u\|_{L^q_tL_x^{r}(I\times\Omega)} \le 2\eta
\right\}
\end{align*}
Following the strategy of \cite{KV13,KVZ16}, we shall show that
$\Phi$ is a contraction on the set $B_1\cap B_2$ under the metric
$$d(u,v):=\|u-v\|_{L^q_tL_x^{r}(I\times\Omega)}.$$

By Theorem \ref{Strichartz}, we have
\begin{align}\label{e5.1}
\left\|(-\Delta_{\Omega})^{s/2}\Phi(u)\right\|_{L^q_tL^r_x(I\times \Omega)}
&\le \left\|(-\Delta_{\Omega})^{s/2}e^{it\Delta_{\Omega}}u_0\right\|_{L^q_tL^r_x(I\times \Omega)}
+C\left\|(-\Delta_{\Omega})^{s/2} (|u|^pu)\right\|_{L^{\tilde q'}_tL^{\tilde r'}_x(I\times \Omega)}\\ \nonumber
&\le \eta+C\left\|(-\Delta_{\Omega})^{s/2} (|u|^pu)\right\|_{L^{\tilde q'}_tL^{\tilde r'}_x(I\times \Omega)}.
\end{align}
By applying  the chain rule obtained in Corollary \ref{c1.1} and the H\"older inequality, we see that
\begin{align}\label{e5.2}
\left\|(-\Delta_{\Omega})^{s/2} (|u|^pu)\right\|_{L^{\tilde q'}_tL^{\tilde r'}_x(I\times \Omega)}&=
\left(\int_I \left(\int_\Omega |(-\Delta_{\Omega})^{\frac{s}{2}} (|u|^pu)|^{\tilde r'}\,dx\right)^{\tilde q'/\tilde r'}\,dt\right)^{1/\tilde q'}\\ \nonumber
&\lesssim \left(\int_I \left(\int_\Omega |(-\Delta_{\Omega})^{\frac{s}{2}}u|^{r}\,dx\right)^{\frac{\tilde q'}{r}}\left(\int_\Omega |u|^{\frac{pr\tilde r'}{r-\tilde r'}}\,dx\right)^{\frac{\tilde q'(r-\tilde r')}{r\tilde r'}}\,dt\right)^{1/\tilde q'}\\ \nonumber
&\lesssim \left\|(-\Delta_{\Omega})^{s/2}u\right\|_{L^{ q}_tL^{ r}_x(I\times \Omega)}\left(\int_I\left(\int_\Omega |u|^{\frac{pr\tilde r'}{r-\tilde r'}}\,dx\right)^{\frac{\tilde q'(r-\tilde r')}{r\tilde r'}\frac{q}{q-\tilde q'}}\,dt\right)^{\frac{q-\tilde q'}{q\tilde q'}}.
\end{align}
Notice that
\begin{align*}
  \frac{pr\tilde r'}{r-\tilde r'}=\frac{4r}{2-2s}  \frac{2}{r-2}=\frac{2r}{2-sr}
\end{align*}
and
\begin{align*}
  \frac{pq\tilde{q}'}{q-\tilde{q}'}=\frac{pq}{q-1}=q.
\end{align*}
By applying the Sobolev embedding, we further see that
\begin{align}\label{e5.3}
\left(\int_I\left(\int_\Omega |u|^{\frac{pr\tilde r'}{r-\tilde r'}}\,dx\right)^{\frac{\tilde q'(r-\tilde r')}{r\tilde r'}
\frac{q}{q-\tilde q'}}\,dt\right)^{\frac{q-\tilde q'}{q\tilde q'}}&\lesssim
\left(\int_I\left(\int_\Omega |(-\Delta_{\Omega})^{\frac s2}u|^{r}\,dx\right)^{\frac{pq\tilde q'}{r(q-\tilde q')}}\,dt\right)^{\frac{q-\tilde q'}{q\tilde q'}}\\ \nonumber
&\lesssim \left\|(-\Delta_{\Omega})^{s/2}u\right\|_{L^{ q}_tL^{ r}_x(I\times \Omega)}^p.
\end{align}
Thus, from \eqref{e5.2} and \eqref{e5.3}, it follows that
\begin{align}\label{e5.4}
\left\|(-\Delta_{\Omega})^{s/2} (|u|^pu)\right\|_{L^{\tilde q'}_tL^{\tilde r'}_x(I\times \Omega)}&\lesssim
\left\|(-\Delta_{\Omega})^{s/2}u\right\|_{L^{ q}_tL^{ r}_x(I\times \Omega)}^{1+p}\lesssim \eta^{1+p}.
\end{align}
Consequently, by \eqref{e5.1} and \eqref{e5.4}, we see that
\begin{align}\label{e5.5}
\left\|(-\Delta_{\Omega})^{s/2}\Phi(u)\right\|_{L^q_tL^r_x(I\times \Omega)}
&\le  \eta+C(\eta)^{1+p}\le 2\eta,
\end{align}
provided $\eta$ small enough.

Similarly, by applying Theorem \ref{Strichartz}, the chain rule and the Sobolev embedding, we have
\begin{align}\label{e5.6}
\left\|\Phi(u)\right\|_{L^\infty_tH^{s}_x(I\times \Omega)}
&\le \left\|u_0\right\|_{H^{s}_x(\Omega)}
+C\left\||u|^pu\right\|_{L^{\tilde q'}_tH^{\tilde r'}_x(I\times \Omega)}\\ \nonumber
&\le \left\|u_0\right\|_{H^{s}_x(\Omega)}+C\left\|u\right\|_{L^{ q}_tH^{s}_x(I\times \Omega)}\left\|(-\Delta_{\Omega})^{s/2} u\right\|^p_{L^{q}_tL^{r}_x(I\times \Omega)}\\ \nonumber
&\le \left\|u_0\right\|_{H^{s}_x(\Omega)}+C\left(2C\left\|u_0\right\|_{H^{s}_x(\Omega)}+2\eta\right) (2\eta)^p\\ \nonumber
&\le  2\left\|u_0\right\|_{H^{s}_x(\Omega)}+C(2\eta)^{1+p}
\end{align}
and
\begin{align}\label{e5.7}
\left\|\Phi(u)\right\|_{L^q_tL^{r}_x(I\times \Omega)}
&\le C\left\|u_0\right\|_{L^2_x(\Omega)}
+C\left\||u|^pu\right\|_{L^{\tilde q'}_tL^{\tilde r'}_x(I\times \Omega)}\\ \nonumber
&\le  C\left\|u_0\right\|_{L^2_x(\Omega)}+C\left\|u\right\|_{L^{ q}_tL^{ r}_x(I\times \Omega)}\left\|(-\Delta_{\Omega})^{s/2} u\right\|^p_{L^{q}_tL^{r}_x(I\times \Omega)}\\ \nonumber
&\le \left\|u_0\right\|_{L^2_x(\Omega)}+C\left(2C\left\|u_0\right\|_{L^2_x(\Omega)}\right) (2\eta)^p\\ \nonumber
&\le  2C\left\|u_0\right\|_{L^2_x(\Omega)},
\end{align}
provided $\eta$ small enough.

Therefore, from \eqref{e5.5}, \eqref{e5.6}, and \eqref{e5.7}, it follows that
$\Phi$ maps $B_1\cap B_2$ to itself. Finally, note that the same arguments as above yield that
\begin{align*}
d(\Phi(u),\Phi(v))&\le C\left\||u|^pu-|v|^pv\right\|_{L^{\tilde q'}_tL^{\tilde r'}_x(I\times \Omega)}\nonumber\\
&\le Cd(u,v) \left(\left\|(-\Delta_{\Omega})^{s/2} u\right\|^p_{L^{q}_tL^{r}_x(I\times \Omega)}+
\left\|(-\Delta_{\Omega})^{s/2} v\right\|^p_{L^{q}_tL^{r}_x(I\times \Omega)}\right)\nonumber\\
&\le Cd(u,v)(2\eta)^{p}\nonumber\\
&\le \frac 12d(u,v),
\end{align*}
for small enough $\eta$. The proof is therefore complete.
\end{proof}

\subsection*{Acknowledgments}
\addcontentsline{toc}{section}{Acknowledgments} \hskip\parindent
R. Jiang would like to thank Professor Changxing Miao for kindly explaining the NLS equation to him and thank Professor Yuan Zhou for helpful discussions. R. Jiang was partially supported by NNSF of China (12471094 \& 11922114). S. Yang is partially supported by NNSF of
China (Grant No. 12431006), the Key Project of Gansu Provincial National
Science Foundation (Grant No. 23JRRA1022)
and the Innovative Groups of Basic Research in Gansu Province (Grant No. 22JR5RA391).

\noindent Renjin Jiang \& Houkun Zhang\\
\noindent Academy for Multidisciplinary Studies\\
\noindent Capital Normal University\\
\noindent Beijing 100048\\
\noindent \texttt{rejiang@cnu.edu.cn}\\
\noindent \texttt{2230501010@cnu.edu.cn}

\

\noindent Tianjun Shen\\
\noindent School of Mathematical Sciences\\
\noindent Beijing Normal University\\
\noindent Beijing 100082\\
\noindent  \texttt{shentj@bnu.edu.cn}

\

\noindent Sibei Yang\\
\noindent School of Mathematics and Statistics\\
\noindent Gansu Key Laboratory of Applied Mathematics and Complex Systems\\
\noindent  Lanzhou University\\
\noindent Lanzhou 730000\\
\noindent People's Republic of China\\
\noindent \texttt{yangsb@lzu.edu.cn}

\end{document}